\documentclass[11pt]{article}

\usepackage{amsmath, amsthm, amssymb, amsfonts, enumerate}
\usepackage{graphicx}
\usepackage{wasysym, pdfpages}
\usepackage[colorlinks=true,linkcolor=blue,urlcolor=blue,citecolor=red]{hyperref}
\usepackage{geometry}
\usepackage{dsfont}
\usepackage[titles]{tocloft}
\usepackage{multirow}

\def\p{\vskip4truept \noindent}

\title{Existence of the uniform value in repeated games with a more informed controller\\
\small{November 29th 2012}}
\numberwithin{equation}{section}
\newtheorem{theoreme}{Theorem}[section]
\newtheorem{proposition}[theoreme]{Proposition}

\newtheorem{hypothese}{Hypotheses}[section]
\newtheorem{remarque}[theoreme]{Remark}

\newtheorem{lemme}[theoreme]{Lemma}
\newtheorem{corollaire}[theoreme]{Corollary}
\newtheorem{definition}[theoreme]{Definition}

\newtheorem{example}[theoreme]{Example}

\newcommand{\RR}{\ensuremath{\mathbb R}}
\newcommand{\PP}{\ensuremath{\mathbb P}}

\newcommand{\EE}{\ensuremath{\mathbb E}}
\newcommand{\NN}{\ensuremath{\mathbb N}}
\newcommand{\A}{\ensuremath{\mathcal A}}

\newcommand{\CL}{\ensuremath{\mathcal L}}

\newcommand{\T}{\ensuremath{\mathcal T}}

\newcommand{\F}{\ensuremath{\mathcal F}}

\newcommand{\HH}{\ensuremath{\mathcal H}}
\newcommand{\ind}{\ensuremath{\mathds 1}}
\newcommand{\G}{\ensuremath{\mathcal G}}

\newcommand{\De}{\Delta}
\newcommand{\DeI}{\Delta(\mathcal{I})}

\newcommand{\DeK}{\Delta(K)}

\newcommand{\DDDK}{\Delta_f(\Delta_f(\Delta(K)))}
\newcommand{\la}{\lambda}
\newcommand{\de}{\delta}

\newcommand{\si}{\sigma}
\newcommand{\ga}{\gamma}
\newcommand{\Ga}{\Gamma}

\newcommand{\limla}{\lim_{\lambda \to 0}}
\newcommand{\limn}{ \lim_{n\to\infty}}
\newcommand{\Tau}{\text{\Large $\tau$}}

\let\inf\relax \DeclareMathOperator*\inf{\vphantom{p}inf}
 
\let\min\relax \DeclareMathOperator*\min{\vphantom{p}min}

\def\abstract{\begin{center} \small\bf Abstract\end{center}\small}

\date{}
 \author{Fabien Gensbittel\thanks{TSE (GREMAQ, Universit\' e Toulouse 1 Capitole), France. \href{mailto:fabien.gensbittel@tse-fr.eu}{fabien.gensbittel@tse-fr.eu}}, Miquel Oliu-Barton\thanks{Equipe Combinatoire et Optimisation, Universit\'{e} Paris VI, France. \href{mailto:miquel.oliu.barton@normalesup.org}{miquel.oliu.barton@normalesup.org}}, Xavier Venel\thanks{School of Mathematical Sciences, Tel Aviv University, Tel Aviv 69978, Israel.  \href{mailto:xavier.venel@sip.univ-tlse1.fr}{xavier.venel@sip.univ-tlse1.fr}}}

\begin{document}
\maketitle

\begin{abstract}
We prove that in a general zero-sum repeated game where the first player is more informed than the second player and controls the evolution of information on the state, the uniform value exists.
This result extends previous results on Markov decision processes with partial observation (Rosenberg, Solan, Vieille 
\cite{rsv2002}), and repeated games with an informed controller (Renault \cite{R2012}). Our formal definition of a more informed player is more general than the inclusion of signals, allowing therefore for imperfect monitoring of actions. We construct an auxiliary stochastic game whose state space is the set of second order beliefs of player $2$ (beliefs about beliefs of player $1$ on the true state variable of the initial game) with perfect monitoring and we prove it has a value by using a result of Renault \cite{R2012}. A key element in this work is to prove that player $1$ can use strategies of the auxiliary game in the initial game in our general framework, which allows to deduce that the value of the auxiliary game is also the value of our initial repeated game by using classical arguments.
\end{abstract}

\vspace{6cm}
\noindent\textbf{Acknowledgements :} 
The authors gratefully acknowledge the support of the Agence Nationale de la Recherche, under grant ANR JEUDY, ANR-10-BLAN 0112.
The third author acknowledges the support of the the Israel Science Foundation under Grant $\sharp$1517/11.
Part of this work was done when the third author was Ph.D. student at the Universit\'{e} Toulouse 1 Capitole.

\newpage
\setcounter{tocdepth}{2}
\tableofcontents

\section{Introduction}

Zero-sum repeated games with incomplete information were introduced by Aumann and Maschler in 1966 \cite{aumasch} in order to study repeated interactions between two players having a different information. The authors also introduced a notion of value for these games usually called uniform value and proved its existence for games with incomplete information on one side. Mertens and Neyman \cite{mn} proved that the uniform value exists for finite stochastic games and several works were devoted since then to prove the existence of the uniform value for some subclasses of the general model of repeated games. 
Recently, Renault proved in \cite{R2012} that the uniform value exists in repeated games with an informed controller using an approach based on an existence result for dynamic programming problems (Renault, \cite{R2011}). The existence theorem in \cite{R2012} requires that the first player observes the state variable at each stage and controls and observes the evolution of the beliefs of the second player on the state variable. 
\p
In the present work, we prove that the uniform value exists in the class of repeated games with a more informed controller. Our existence result requires that the first player is more informed about the state variable than the second player and also that he controls the evolution of beliefs of the second player. A weaker version of our result was conjectured in the conclusion of \cite{R2012}, and it was suggested that the proof may be based on an auxiliary game whose state space would be the pair of beliefs of both players about the original state variable. We show that the analysis requires actually to introduce an auxiliary game whose state space is the set of second order beliefs of the less informed player and provide a set of weaker assumptions than those suggested in \cite{R2012}, allowing to deal with imperfect monitoring of actions.  
\p
 The paper is organized as follows:  In section \ref{model}, we describe the general model of repeated games and introduce three assumptions that formalize the notion of a more informed controller. In section \ref{app}, we check that several models previously studied in the literature  satisfy these three assumptions. Section \ref{discuss} is dedicated to a discussion of the assumptions and a precise study of their implications. In addition, we provide a second version of the theorem with stronger, but easier to check, assumptions. The last section \ref{proof} is dedicated to the proof of existence of the uniform value. We introduce there an auxiliary stochastic game with perfect monitoring on an auxiliary state variable which represents the beliefs of player $2$ about the beliefs of player $1$ about the state variable of the original game. We prove that this auxiliary game has a uniform value using the main theorem of Renault \cite{R2012} and that player $1$  can use optimal strategies in this auxiliary  stochastic game in order to play optimally in the original repeated game. Finally, we prove that player $2$ can also guarantee this value by playing by blocks, so that both games have a uniform value and these values are equal.
\section{Model}\label{model}
\subsection{General definitions and notation}
For any metric space $X$, let $\De(X)$ denote the set of Borel probability distributions on $X$. If $X$ is a finite set (endowed with the discrete metric) of cardinal $|X|$, then $\De(X)$ is precisely the $|X|$-dimensional simplex.
$\De_f(X)\subset \De(X)$ denotes the probability distributions supported on a finite subset of $X$ and  $\de_x$ denotes the Dirac measure on $x\in X$.

A zero-sum repeated game is described by a $8$-tuple $(K,I,J,g,C,D,\pi,q)$, where  $K$ is the state space, $I$ and $J$ are the action sets for player $1$ and $2$ respectively, $g$ is a payoff function $g:K\times I\times J \to [0,1]$, and $C$ 
and $D$ are the signal sets for player $1$ and $2$ respectively. $\pi \in \Delta_f(K\times \NN  \times \NN)$ denotes the initial probability and $q:K \times I \times J \to \De(K\times C\times D)$ denotes the transition function.

The game is played as follows: At the beginning of the game, the triple $(k_1,c_1,d_1)$ is chosen according to the initial probability distribution $\pi\in \De_f(K\times \NN \times \NN)$. For each stage $m \geq 1$, player $1$ observes the signal
$c_m$ and player $2$ observes the signal $d_m $. Then both players choose 
actions $(i_m,j_m)\in I\times J$ based on their own past actions and on the sequence of signals they observed
(i.e. we assume perfect recall). 
Given the state $k_m$ and the actions $(i_m,j_m)$, a new triple $(k_{m+1},c_{m+1},d_{m+1}) \in K\times C\times D$ is chosen according to the probability distribution $q(k_{m},i_m,j_m)$. The payoff for stage $m$ of player 1 is $g(k_m,i_m,j_m)$ and
the game proceeds to stage $m+1$. The stage payoffs are not directly observed by the players and  cannot be deduced, in general, from their observations. The sets of initial signals can be any finite subset of $\NN$. This generalization is for
technical reasons only. Indeed, it will 
very convenient in the sequel to consider this possibly larger set of initial signals in order to have a simple way to deal with the recursive structure of the game.

The information held by player $1$  before his play at stage $m$, called player $1$'s  private history, is given by 
\[ h^I_m \triangleq (c_1,i_1,\dots,c_{m-1},i_{m-1},c_m) \in \NN \times (I\times C)^{m-1}. \]
Similarly, the information held by player $2$ is represented by  
\[ h^{II}_m \triangleq (d_1,j_1,\dots,d_{m-1},j_{m-1},d_m) \in \NN \times (J \times D)^{m-1}.\]
Let $\HH^{I}$ (resp. $\HH^{II}$) denote the set of all finite private histories for player $1$ (resp. of all private histories for player  $2$). We assume the sets $K$, $I$, $J$, $C$ and $D$ are all finite and that the description of the model 
is common knowledge.

Instead of $\NN$, the sets of initial signals will often be denoted by $C'$ and $D'$ where $C'$ and $D'$ are finite subsets of $\NN$. We will also write abusively that $\pi \in \Delta(K\times C'\times D')$. The initial signals will still be
denoted by $(c_1,d_1)$. Reciprocally, given finite sets $C'$ and $D'$, any $\pi \in \Delta(K\times C'\times D')$ can be seen as an element of $\Delta_f(K\times \NN \times \NN)$ using some enumerations of $C'$ and $D'$. The main advantage is
that any couple of finite private histories can be embedded in $\NN\times \NN$ via some enumerations.
This advantage will become clear in the proof. 

\paragraph{Strategies} 
A behavior strategy for player $1$ is a map from private histories $\HH^{I}$  to probabilities  over $I$ . The set of 
behavior strategies of player $1$ is denoted by $\Sigma$. Every strategy $\si\in \Sigma$ corresponds to a 
sequence $\{\si_m\}_{m\geq 1}$, where $\si_m$ is defined on the set of histories up to stage $m$. That is, 
\[ \si_m:\NN \times (I \times C)^{m-1} \to \De(I). \] 
Similarly, a behavior strategy $\tau$ for player $2$ is a map from private histories $\HH^{II}$ to probability distributions over $J$. The set of behavior strategies of player $2$ is denoted by 
$\T$. Any  $\tau$ corresponds to a 
sequence $\{\tau_m\}_{m\geq 1}$, with 
\[ \tau_m: \NN \times (J\times D)^{m-1} \to \De(J). \]

The initial distribution $\pi$, the transition function $q$ and a behavior strategy profile $(\si,\tau)\in \Sigma\times \T$ induce a unique probability distribution over the set of plays 
$K \times \NN \times \NN \times (K\times C \times D \times I \times J)^{\infty}$, denoted by $\PP^{\pi}_{\si\tau}$. Let $\EE^\pi_{\si\tau}=\EE_{\PP_{\si\tau}^\pi}$ denote the expectation with respect to the probability $\PP^\pi_{\si\tau}$.

\paragraph{Evaluations of the payoff} A second component of the model is the way in which the total payoff of player $1$ is evaluated, in terms of the sequence of stage payoffs $\{g(k_m,i_m,j_m)\}_{m\geq 1}$. The two classical evaluations 
correspond to the $n$-stage  game and the $\la$-discounted game. In the former, the payoff function is the expected Ces\`aro mean of the stage payoffs of the $n$ first stages, i.e.
\[\ga_n(\pi,\si,\tau)=\EE^\pi_{\sigma\tau}\big[\frac{1}{n}\sum\nolimits_{m=1}^n g(k_m,i_m,j_m)\big].\]
 In the latter, the payoff is taken as the expected Abel sum, with respect to the discount factor $0 <\la  \leq 1$, i.e.
\[ \ga_\la(\pi,\si,\tau)=\EE^\pi_{\sigma\tau}\big[\la \sum\nolimits_{m\geq 1}(1-\la)^{m-1} g(k_m,i_m,j_m)\big].\] 
More generally, one may consider any \textit{compact evaluation}. That is, for any $\theta \in\Delta(\NN^*)$, let 
\begin{eqnarray}\label{ev}
\ga_{\theta}(\pi, \si, \tau)&=&\EE^\pi_{\sigma\tau}\big[\sum\nolimits_{m \geq 1}\theta_m  g(k_m,i_m,j_m)\big].
\end{eqnarray}
Denote by $\Gamma_\theta(\pi)$ the $8$-tuple defined above together with the $\theta$-evaluation.

\paragraph{The value function}
For any $\pi\in \Delta_f(K \times \NN \times \NN)$ and any $\theta \in \Delta(\NN^*)$, $\Gamma_{\theta}(\pi)$ is known to have a value, denoted by $v_\theta(\pi)$. It satisfies 
\[ v_\theta(\pi)=\sup_{\sigma \in \Sigma} \inf_{\tau \in \T} \ga_{\theta}(\pi, \si, \tau)=\inf_{\tau \in \T} \sup_{\sigma \in \Sigma} \ga_{\theta}(\pi, \si, \tau).\]

\begin{remarque}
These general evaluations will be used in section \ref{proof}, and we will only need to consider  probabilities with finite support (i.e. $\theta \in \Delta_f(\NN^*)$).
\end{remarque}

Let $\si\in \Sigma$ be a behavior strategy and let $h\in \HH^{I}$ be some finite private history of player $1$. We denote by $\si(h)$ the behavior strategy of player $1$ after the history $h$. Equivalently, $\si(h)$ is the restriction of the map
$\si$ to the subset of histories beginning with $h$. In particular, given some strategy profile $(\si,\tau)\in \Sigma\times \T$ and two signals $(c,d)\in C'\times D'$, consider the profile $(\si(c),\tau(d))$. It may be interpreted as a strategy
profile in a game in which the players have no initial signals. More formally, for any $p\in \De(K)$, we will use the following notation:
\[ \gamma_{\theta}(k,\sigma(c),\tau(d))\triangleq \gamma_{\theta}(\delta_{(k,c,d)},\sigma,\tau), \ \text{and}\quad  \gamma_{\theta}(p,\sigma(c),\tau(d))\triangleq \gamma_{\theta}(p\otimes \delta_{(c,d)},\sigma,\tau).\]
With this notation, the payoff can be written as
\[ \gamma_{\theta}(\pi,\sigma,\tau) = \EE_\pi[ \gamma_{\theta}(k,\sigma(c),\tau(d))]=\sum_{(k,c,d)\in K\times C'\times D'}
\gamma_{\theta}(k,\sigma(c),\tau(d)) \pi(k,c,d). \]

Alternatively, one can consider the game as having \textit{per se} infinitely many stages.

\paragraph{Uniform value} The infinitely repeated game is denoted by $\Gamma_\infty(\pi)$. Let us present here some important definitions relative to the game $\Ga_\infty(\pi)$. Its value will be called the  \emph{uniform value}
and denoted by $v_\infty(\pi)$.

\begin{definition}
Let $v$ be a real number,
\begin{itemize}
\item Player $1$ can guarantee $v$ in $\Gamma_{\infty}(\pi)$ if for any $\epsilon >0$ there exists a strategy $\sigma\in \Sigma$ of player $1$ and an integer $N\in \NN$, such that
\[ \forall n\geq N, \ \forall \tau \in \Tau, \ \gamma_n(\pi,\sigma,\tau) \geq v- \epsilon. \]
We say that such a strategy $\sigma$ guarantees $v-\epsilon$ in $\Gamma_\infty(\pi)$ and define
\[
\underline{v}_\infty(\pi) \triangleq \sup \{v\in\RR\ | \ \text{player }1\text{ can garantee } v\}.\]
\item Player $2$ can guarantee $v$ in $\Gamma_\infty(\pi)$ if for any $\epsilon >0$ there exists a strategy $\tau \in \Tau$ of player $2$ and an integer $N\in \NN$, such that
\[ \forall n\geq N, \ \forall \sigma \in \Sigma, \  \gamma_n(\pi,\sigma,\tau) \leq v+ \epsilon.\]
We say that such a strategy $\tau$ guarantees $v+\epsilon$ in $\Gamma_\infty(\pi)$ and define
\[\overline{v}_\infty(\pi) \triangleq \inf \{v\in\RR\ | \ \text{player }2\text{ can garantee } v\}.\]
\item If $\underline{v}_\infty(\pi)=\overline{v}_\infty(\pi)$ the uniform value exists and we denote by $v_\infty(\pi)$ the common value.
\end{itemize}
\end{definition}

The existence of a uniform value $v_\infty$ is stronger than the existence of a limit value (or asymptotic value), in the sense that it implies (see e.g. Neyman and Sorin \cite[Theorem 1]{NS10} for more general evaluations)
\[ \limla v_{\la} =\limn v_n=v_\infty. \]

\subsection{Model with a more informed controller}\label{miss}
We will consider a particular class of the general model presented above, which generalizes both the class of repeated games considered by Renault \cite{R2012} and the model of Partially Observable Markov Decision Processes (see section 3.1). As usual in games with incomplete information, we call belief of player $1$ at stage $m$ about some random variable $\xi$ the conditional law of $\xi$ given the information
held by player $1$ at stage $m$. In the sequel, the first order beliefs of a player denote beliefs about the state variable $k$, and second order beliefs of a player denote beliefs about the first order beliefs of his opponent. 
\p
We assume the following three hypotheses at every stage $m$ of the game:
\begin{itemize}
\item[$(a1)$] Player $1$'s first order belief is more accurate than player $2$'s first order belief.
\item[$(a2)$] Player $1$ can compute the second order beliefs of player $2$.
\item[$(a3)$] Player $1$ controls the evolution of second order beliefs of player $2$.
\end{itemize}
The main result of this paper is to establish the existence of the uniform value under these assumptions. Let the formal transcription
of $(a1)-(a3)$ defined below be denoted by $(A1)-(A3)$.

\begin{theoreme}\label{main1}
Let $\Ga$ be a repeated game with a more informed controller, i.e. such that assumptions $(A1)$, $(A2)$, $(A3)$ hold. Then the uniform value exists.
\end{theoreme}

\begin{remarque} It was already pointed out in the literature (see e.g. Mertens \cite{mertensicm}) that in games with a more informed player, the analysis of beliefs can be restricted to second order beliefs of the less informed player. In this work, the definition of more informed is slightly more general than the inclusion of signals and a similar reduction is made formally in Lemma \ref{projection}. 
\end{remarque}

\subsection{Formal assumptions}

Let us present here a rigorous transcription of the informal assumptions $(a1)-(a3)$. In the next section, we will present some of the models which satisfy our three assumptions and to which, consequently, Theorem \ref{main1} applies.
\p
Let us start with some notations: 
\p
Given  some probability distribution $\mu \in \De_f(X\times Y)$  over
a product, we denote by 
\[\mu(x)\triangleq \sum_{y\in Y} \mu(x,y).\]
For any random variable $\xi$ defined on a probability space $(\Omega,\A,\PP)$ and $\F$ a sub $\sigma$-algebra of $\A$, let $\CL_{\PP}(\xi\mid \F)$ denote the conditional distribution of $\xi$ given $\F$,
which is seen as a $\F$-measurable random variable\footnote{All random variables appearing here take only finitely many values so that the definition of conditional laws does not require any additional care
about measurability.} and let $\CL_{\PP}(\xi)$ denote the distribution of $\xi$.
\p
In the sequel, both functions $g$ and $q$ are linearly extended to $\De(K) \times I \times J$.
\p
Assumption $(a1)$ can now be formalized as follows:
\begin{itemize}
\item[$(A1)$] $\forall n\geq 1, \ \forall (\si,\tau)\in \Sigma \times \T, \quad \CL_{\PP^\pi_{\si\tau}}(k_{n} \mid h_n^I,h_n^{II}) = \CL_{\PP^\pi_{\si\tau}}(k_n \mid h_n^I).$
\end{itemize}
In words, at every stage and given any strategy profile, player $2$'s information does not contain any information about the state variable that is not already contained in player $1$'s information. 
Assumption $(A1)$ is equivalent to the conditional independence of $k_{n}$ and $h^{II}_n$, given $h^I_n$, under the probability $\PP_{\si\tau}^\pi$.

\p
For $n=1$, this equation does not depend on $\si$ and $\tau$ and it can be reformulated as
\begin{itemize}
\item[$(A1a)$] $\pi(c)\pi(k,c,d)=\pi(k,c)\pi(c,d),\ \forall (k,c,d)\in K \times C'\times D'$.
\end{itemize}

In order to model the players' information about the state variable at stage $n$, we need to define three variables $x_n$, $y_n$ and $\eta_n$. Before choosing their first action, the players receive signals 
$(c_1,d_1)\in C'\times D'$. The (random) variable 
\[ x_1\triangleq  \CL_\pi (k_1\mid c_1) \in \De(K)\]
represents the \emph{first order beliefs of player $1$} about the initial state. Let $x_1(c_1)\in \De(K)$ denote its realization,  i.e. the beliefs of player $1$ once he has received the signal $c_1\in C'$. Thus, $x_1(c_1)=\CL_\pi(k_1|c_1)$ and each signal $c_1 \in C'$ occurs with probability $\pi(c_1)$, so that $\CL_\pi(x_1)=\sum_{c_1\in C'}\pi(c_1) \de_{x_1(c_1)}$. 
Similarly, define the \emph{second order beliefs of player $2$}, i.e. beliefs about player $1$'s beliefs about the initial state 
\[ y_1\triangleq  \CL_\pi(x_1 | d_1 ) \in \Delta_f(\Delta(K).\] 
With probability $\pi(d_1)$, player $2$'s beliefs about player $1$'s beliefs (about the state variable) are distributed as follows:  
\[ y_1(d_1)=\sum_{c_1\in C'} \pi(c_1|d_1)\de_{x_1(c_1)}\in \De_f(\De(K)),\]
with a slight abuse of notations since we write $\pi(c_1|d_1)$ instead of $\pi(c_1=c|d_1)$ with a sum over $c \in C'$. 
Finally, let $\eta_1$ be the distribution of the second order  beliefs of player $2$
\[ \eta_1 \triangleq \CL_\pi(y_1)= \sum_{d_1 \in D'} \pi(d_1) \de_{y(d_1)}\in \De_f(\De_f(\De(K))).\]
Notice that the Dirac measures involved in the definition of $\CL_\pi(x_1)$ or of $\CL_\pi(y_1)$ refer to different spaces: the former refers to $\De(K)$, the latter to $\De_f(\De(K))$. 
\p
More generally, for some fixed strategy profile $(\si,\tau)\in \Sigma\times \T$, let us denote the first order beliefs of player $1$ at stage $n$ by $x_n\in \De(K)$, the second order beliefs of player $2$ at stage 
$n$ by $y_n\in \De_f(\De(K))$, and the distribution of $y_n$ by $\eta_n$.

\begin{definition} Put $x_{n} \triangleq  \CL_{\PP_{\si\tau}^\pi} (k_{n}\mid h^I_n)$, $y_n\triangleq  \CL_{\PP^\pi_{\sigma\tau}}(x_n | h_{n}^{II})$, and $\eta_n\triangleq  \CL_{\PP^\pi_{\sigma\tau}}(y_n)$. 
\end{definition}

Let us illustrate these definitions through the following example. 
\begin{example} \label{varaux}
Let $K=\{k_1,k_2\}$ be set of states space, $U=\{u_1,u_2\}$ a set of public signals 
and 
$S=\{s_1, s_2,s_3 \}$ a set of private signals for player $1$. 
Using the notations above, let $C=U\times S$ (resp. $D=U$) be the set of signals for player $2$ (resp. $2$). 
We consider $\pi \in \De(K \times S\times U)$ defined by 
\[ \begin{matrix}  
 & \begin{matrix} u_1 & u_2  \end{matrix}& &\begin{matrix} u_1 & u_2  \end{matrix} \vspace{.2cm} \\
\begin{matrix} s_1 \\ s_2 \\ s_3 \end{matrix}  &\begin{pmatrix} \frac{8}{24} & 0 \\ \frac{1}{24} & \frac{3}{24} \\ 0 & 0 \\ \end{pmatrix} 
 & \begin{matrix} s_1 \\ s_2 \\ s_3 \end{matrix} 
 & \begin{pmatrix} 0 & 0 \\ \frac{1}{24} & \frac{3}{24} \\ \frac{2}{24} & \frac{6}{24} \\ \end{pmatrix} \vspace{.2cm} \\ 
 & k_1 & & k_2 
 \end{matrix}. \]
It is more convenient here to use $K\times S \times U$ but $\pi$ can be understood as a probability on $K \times C\times D$. To simplify notations, let us identify $\Delta(K)$ with $[0,1]$ with the convention 
that $p \in \Delta(K)$ is identified with $p(k_1)$. If player $1$ receives signal $(s_1, u_1)$, then  $x_1 = 1$. If he receives $(s_2, u_1)$ or $(s_2, u_2)$, then $x_1=\frac{1}{2}$. Finally, if he receives $(s_3,u_1)$ 
or $(s_3, u_2)$, then $x_1=0$. The value of $x_1$ depends only on his private signal.
\p
We now compute the second order beliefs of player $2$. If player $2$ receives $u_1$ then his beliefs about the private signal of player $1$ are $\frac{8}{24} \delta_{s_1}+ \frac{2}{24} \delta_{s_2} + \frac{2}{24} \delta_{s_3}$, so that
\[ y_1(u_1) =\frac{8}{24} \delta_{1}+ \frac{2}{24} \delta_{\frac{1}{2}} + \frac{2}{24} \delta_{0}.\]
If player $2$ receives $u_2$, then we obtain 
\[ y_1(u_2) = \frac{6}{24} \delta_{\frac{1}{2}} + \frac{6}{24} \delta_{0}.\] 
To conclude, player $2$ receives each signal with probability $\frac{1}{2}$, so that $\eta_1$ is equal to
\[ \eta_1=\frac{1}{2} \delta_{\frac{8}{24} \delta_{1}+ \frac{2}{24} \delta_{\frac{1}{2}} + \frac{2}{24} \delta_{0}} + \frac{1}{2} \delta_{\frac{6}{24} \delta_{\frac{1}{2}} + \frac{6}{24} \delta_{0}}.\]
\end{example}

Assumption $(a2)$ will be split in two parts $(A2a)$ and $(A2b)$. At first, we assume that player $1$ is able to compute the variable $y_1$, which is a constraint on the initial probability $\pi$ only.

\begin{itemize}
\item[$(A2a)$] There exists a map $f_1=f_1^\pi: C'\to \De(\De(K))$ such that $y_1= f_1(c_1), \ \pi \text{-almost surely}$.
\end{itemize}

Assuming $(A2a)$, we can introduce a special class of strategies for player $1$ which will be needed for the second part of the formal assumption.
\begin{definition}
If $\pi \in \Delta_f(K \times \NN \times \NN)$ fulfills $(A2a)$, a strategy $\si \in \Sigma$ is called \emph{a reduced strategy} if it depends on the initial signal $c_1$ in $C'$ only through $(x_1,y_1)$. Let $\Sigma'(\pi) \subset \Sigma$ denote the subset of reduced strategies.
\end{definition}

The second part of $(A2)$ requires that when player $1$ is using a reduced strategy, the variable $y_2$ has to be $h^{I}_2$-measurable. Formally: 
\begin{itemize}
\item[$(A2b)$]
$\forall \pi \in \Delta_f(K \times \NN \times \NN)$ satisfying $(A2a)$,
$\forall \si \in \Sigma'(\pi),\forall \tau  \in \T, \; \exists f_2=f_2^{\pi,\si,\tau}: \NN \times I \times C \to \De(\De(K))$  such that
$y_2=f_2(c_1,i_1,c_2),\; \PP^{\pi}_{\si,\tau}$-almost surely.
\end{itemize}

The introduction of reduced strategies for player $1$ is necessary in order to exclude non relevant correlations
between players (see example \ref{reducednecessity} in section \ref{discuss}). It will be shown in Lemma 
\ref{projection} and in the proof of the main Theorem that there is no loss in restricting player $1$ to reduced strategies.

In order to state the last assumption, we reduce the set of initial probabilities.
\begin{definition}
Let $\De^*_f(K\times \NN\times \NN)$ be the set of probability distributions satisfying $(A1a)$ and $(A2a)$.
\end{definition}

Assumption $(a3)$ can now be formalized as
\begin{itemize}
\item[$(A3)$] $\forall \pi \in \De^*_f(K\times \NN\times \NN),\  \forall \si\in \Sigma'(\pi), \  \eta_2 \text{ is independent of }\tau \in \T$.
\end{itemize}

\begin{remarque}
Assumptions $(A1,A2)$ imply that the properties $(A1a)$ and $(A2a)$ of the initial probability $\pi$ are preserved
by the transition when player $1$ plays reduced strategies.  Precisely, for all $(\si,\tau)$ with $\si$ reduced,
the law of $(k_2,h^I_2,h^{II}_2)$ under $\PP_{\si\tau}^\pi$, seen as an element of $\De_f(K\times \NN \times \NN)$, belongs to the set $ \De_f^*(K\times \NN \times \NN)$. We will prove in section \ref{discuss} that even if the
two last assumptions $(A2b)$ and $(A3)$ are stated in terms of $y_2$ and $\eta_2$, it is possible to extend these properties by induction for $y_n$ and $\eta_n$ for appropriate strategies. Thus the formal assumptions are coherent with the informal
assumptions. In particular, player $1$ can compute the auxiliary variables $(x_2,y_2,\eta_2)$ without knowing the strategy of player $2$ and therefore play again a reduced strategy at the second stage (i.e. which 
depends  only on $(x_2,y_2)$).
\end{remarque}

\section{Applications}\label{app}

We present in this section several models which satisfy our assumptions. 

\subsection{Partially Observable Markov Decision Processes.}
A POMDP is a one-player game, given by a tuple $(K, I,C,g,q,\pi)$, where $K$ is the state space, $I$ is the action set, $C$ is the signals set, $g:K \times  I \to [0,1]$ is the payoff function, $q:K \times  I\to \De(K\times C)$ 
is the transition function and $\pi$ is an initial distribution on $K \times \NN$. In the finite framework, the existence of the uniform value has been proven by Rosenberg, Solan and Vieille \cite{rsv2002} and it was extended by
Renault \cite{R2011} to arbitrary set of actions and signals with the additional assumption that all the probabilities appearing in the transition or in the definition of strategies have finite support. We will only consider here the finite case.

Formally, a POMDP can be seen  as a repeated game in which player $2$ is dummy (i.e. his action set $J$ is a singleton). Since player $2$ has only one action, his information plays no role here. The assumptions $(A1)-(A3)$ hold obviously.

\subsection{Repeated game with a perfectly informed controller.}

The model of a repeated game with an informed controller introduced by Renault \cite{R2012} fulfils our assumptions. In this model, player $1$ is perfectly informed of the state and of the signal of player $2$, in the sense
that he can deduce the true state variable and the signal of player $2$ from his signals. Moreover, the transition $q$ is such that player $2$ has no influence on the joint distribution of the pair made by the state variable and his signal.

In \cite{R2012},  the sets of initial signals are $C'=C$ and $D'=D$. Formally, the first assumption (i.e. that player $1$ is
perfectly informed of the state and of the signal of player $2$) is given by
\begin{itemize}
\item[$(HA')$] There exists two mappings $\hat{k}:C \rightarrow K$ and $\hat{d}:C \rightarrow D$ such that, if $E$ denotes $\{(k,c,d)\in K \times C \times D,\, \hat{k}(c)=k,\, \hat{d}(c)=d \}$, then
\[ \pi(E)=1, \text{  and  } q(k,i,j)(E)=1, \ \forall (k,i,j)\in K\times I \times J. \]
\end{itemize}
The second assumption is formalized by
\begin{itemize}
\item[$(HB')$] Player $1$ controls the transition in the sense that the marginal of the transition $q$ on $K \times D$ does not depend on player $2$'s action. 
For $k\in K$, $i\in I$, $j \in J$, we denote by 
$\overline{q}(k,i)$ the marginal of $q(k,i,j)$ on $K\times D$.
\end{itemize}
Let us check that this model satisfies our assumptions.  Assuming that the initial distribution $\pi \in \Delta(K\times C \times D)$ fulfils assumption $(HA')$, we have
\[ y_1(d_1)=\sum_{c_1 \in C} \pi(c_1| d_1) \delta_{\CL_{\pi}(k_1 |c_1,d_1)}=\sum_{c_1 \in C} \pi(c_1 | d_1) \delta_{\delta_{\hat{k}(c_1)}}= \sum_{k_1 \in K} \pi(k_1|d_1) \delta_{\delta_{k_1}}.\]
We deduce that $\pi$ can be seen as an element of $\Delta^*_f(K \times \NN \times \NN)$. Formally, we have to verify our assumptions, 
starting from any initial distribution in $\Delta^*_f(K \times \NN \times \NN)$.
From now on, initial signals $(c_1,d_1)$ belong to arbitrary finite subsets of $\NN$ denoted by $C',D'$
as in the previous section\footnote{One may easily reduce the analysis to a smaller set of initial probabilities, but we chose 
to keep this general formulation since the reduction does not really simplify the proofs.}.   

First, note that any stage $m\geq 2$, player $1$'s first order belief at each stage is a Dirac mass on the current state (i.e. $x_n=\de_{k_n}$).
Thus, adding the signal of player $2$ to the signal of player $1$ does not change the beliefs of the latter, 
which proves $(A1)$. 
It also implies that the second order beliefs of player $2$  can be identified with the first order beliefs of player $2$.
Let $\pi \in \Delta^*_f(K \times \NN \times \NN)$, $\tau$ be a strategy for player $2$ and $\sigma$ a reduced strategy for player $1$.  Recall that $\sigma$ is function only of $(x_1,y_1)$ which are by assumption $c_1$ measurable and that that $y_1$ is $d_1$-measurable, so that there exist two functions $h_1$ and $f_1$ such that, with probability $1$, $\CL_{\pi}(x_1 | d_1)=y_1=h_1(d_1)=f_1(c_1)$.  It follows that 
\[\PP^{\pi}_{\sigma \tau}(k_1,x_1,d_1,i_1,j_1,k_2,d_2)=\pi(d_1)h_1(d_1)(x_1)x_1(k_1)\sigma(x_1,h_1(d_1))(i_1)
\tau(d_1)(j_1)\overline{q}(k_1,i_1)(k_2,d_2).\]
We deduce that
\begin{align*}
y_2(d_1,j_1,d_2) & =\sum_{c_1,i_1,c_2} \PP^{\pi}_{\sigma \tau}(c_1,i_1,c_2| d_1,j_1,d_2) \delta_{\CL_{\PP^{\pi}_{\sigma \tau}}(k_2|c_1,i_1,c_2)} \\
 & =\sum_{c_1,i_1,c_2} \PP^{\pi}_{\sigma \tau}(c_1,i_1,c_2| d_1,j_1,d_2) \delta_{\delta_{\hat{k}(c_2)}}\\
 & =\sum_{k_2} \PP^{\pi}_{\sigma \tau}(k_2| d_1,j_1,d_2) \delta_{\delta_{k_2}}\\
 & =\frac{ \sum_{k_1,x_1,k_2,i_1} \pi(d_1)h_1(d_1)(x_1)\sigma(x_1,h_1(d_1))(i_1)\overline{q}(k_1,i_1)(k_2,d_2)\delta_{\delta_{k_2}} } {\sum_{k_1',x_1',k_2',i_1'} \pi(d_1)y_1(d_1)(x_1')\sigma(x_1',y_1(d_1))(i_1')\overline{q}(k_1',i_1')(k_2',d_2)}\\
& =\frac{ \sum_{k_1,x_1,k_2,i_1} f_1(c_1)(x_1)\sigma(x_1,f_1(c_1))(i_1)\overline{q}(k_1,i_1)(k_2,d_2)\delta_{\delta_{k_2}} } {\sum_{k_1',x_1',k_2',i_1'} f_1(c_1)(x_1')\sigma(x_1',f_1(c_1))(i_1')\overline{q}(k_1',i_1')(k_2',d_2)}\\
\end{align*}
where the above equalities hold almost surely  whenever the conditional probabilities are well-defined.
We deduce that $y_2$ is only a function of $c_1$ and $d_2$ which does not depend on $\tau$, so that the function  $f_2(c_1,c_2)=y_2(c_1,\hat{d}(c_2))$ proves that assumption $(A2)$ is satisfied. Finally the distribution of the random variable $y_2$ is equal to
\[ \eta_2= \sum_{c_1,d_2} \PP^{\pi}_{\sigma \tau}(c_1,d_2) \delta_{y_2(c_1,d_2)}.\]
Using the preceding result and that the marginal of $q$ on $D$ does not depend on the action played
by player $2$, $\eta_2$ does not depend on $\tau$ and $(A3)$ is satisfied.

\begin{remarque}
In a previous work, Renault \cite{R2006} studied the particular case where the state follows a Markov chain $f:K\to \De(K)$, player $1$ observes the state and both players observe the actions. This is easily seen as a 
particular case of the above model. In a more recent work, Neyman \cite{neyman} proved the existence of the uniform value when allowing for any signalling structure on the actions. This last result is not covered by our main theorem since in this case, player $1$ cannot control player $2$'s information about the state variable.
\end{remarque}

\subsection{Player $1$ is more informed about the state.}

In this last paragraph, we assume that actions are observed by both players  after each stage. Moreover, both players receive a public signal in a set $U$, player $1$ receives a signal in a set $S$ and player $2$ has  influence on the joint distribution of the state and signals in $K\times U \times S$.

Formally, it is a repeated game where $C= I\times J \times S \times U$, $D= I \times J \times U$ and the transition function satisfies the following two conditions. At first, the signal $u$ is public and the actions are observed:
\[\forall (k,i,j)\in K \times I \times J, \;  \sum_{k',s,u \in K \times S \times U} q(k,i,j)(k',(i,j,u),(i,j,s,u))=1.\]
Secondly, there exists a function $\hat{q}$ from $K\times I$ to $\Delta(K\times S \times U)$ such that
\[ \forall (k,k',i,j,s,u)\in K \times K \times I \times J \times S \times U,\ \ q(k,i,j)(k',(i,j,s,u),(i,j,u))=\hat{q}(k,i)(k',s,u).\]

Let us stress out that the transition $q$ in itself depends on player $2$ since it has to reveal his actions but as we will see our assumptions are still satisfied.
It was already noticed in Renault \cite{R2012} that it is too restrictive to assume that the transition is fully 
controlled by player $1$. 
This model is a natural generalization of Renault's model, dropping the (important) condition of Player $1$ to know the state at every stage. However, it does not allow for imperfect monitoring of actions as in the previous examples.

Let us check that this model satisfies our assumptions. According to the description of the model, an initial distribution $\pi$ of $(k_1,s_1,u_1)$ 
can be seen as an element of $\Delta^*_f(K \times \NN\times \NN)$ since the signal of player $2$ is contained in the signal of player $1$.
As for the previous example, we will start with a general initial probability $\pi \in \Delta^*_f(K \times \NN\times \NN)$ and initial signals $(c_1,d_1)$.
At first, note that apart from the initial signal $d_1$, histories of player $2$ are contained in histories of player $1$, so that assumption $(A1)$ reduces to
\[\forall n\geq 1, \ \forall (\si,\tau)\in \Sigma \times \T, \quad \CL_{\PP^\pi_{\si\tau}}(k_{n} \mid h_n^I,d_1) =\CL_{\PP^\pi_{\si\tau}}(k_n \mid h_n^I).\]
This property is true for $n=1$ by assumption. Let us proceed by induction on $n$. Assume that $n\geq 2$ and that the property is proved for  $n-1$, i.e. that 
\[ x_{n-1}=  \CL_{\PP^\pi_{\si\tau}}(k_{n-1} \mid h_{n-1}^I,d_1) =\CL_{\PP^\pi_{\si\tau}}(k_{n-1} \mid h_{n-1}^I). \]
Again, let $\hat{q}$ be linearly extended to $\De(K)\times I$. It follows that, 
\[ \CL_{\PP^\pi_{\si\tau}}(k_n,s_n,u_n | h_{n-1}^{I},d_1,i_{n-1},j_{n-1} ) = \hat{q}(x_{n-1},i_{n-1}) = \CL_{\PP^\pi_{\si\tau}}(k_n,s_n,u_n | h_{n-1}^{I},i_{n-1},j_{n-1} ),\]
$(A1)$ follows then directly by disintegration. Moreover, we deduce that 
\begin{equation} \label{exprbelief} 
\PP^{\pi}_{\sigma,\tau}(k_n=k | h_n^{I} ) = \frac{\hat{q}(x_{n-1},i_1)(k,s_n,u_n)}{\sum_{k' \in K}\hat{q}(x_{n-1},i_1)(k',s_n,u_n)}.
\end{equation}
The latter proves that $x_2$ can be expressed as a function of $(x_1,i_1,s_2,u_2)$ which does not depend on $\tau$.
Recall then that by assumption there exist functions $h_1$ and $f_1$ such that with probability $1$, we have
\[y_1=\CL_{\pi}(x_1|d_1)=h_1(d_1)=f_1(c_1).\]
If player $1$ uses a reduced strategy $\sigma$ and player $2$ uses a strategy $\tau$, we have
\[ \PP^{\pi}_{\sigma,\tau} (k_1,x_1,d_1,i_1,j_1,k_2,s_2,u_2)=\pi(d_1)h_1(d_1)(x_1)x_1(k_1)\sigma(x_1,h_1(d_1))(i_1)\tau(d_1)(j_1)\hat{q}(k_1,i_1)(k_2,s_2,u_2).\]
We deduce that
\begin{align*}
y_2(d_1,i_1,j_1,u_2) & =\sum_{x_1,i_1,s_2} \PP^{\pi}_{\sigma,\tau}(x_1,i_1,s_2| d_1,j_1,u_2) \delta_{\CL_{\PP^{\pi}_{\sigma,\tau}}(k_2|x_1,d_1,i_1,j_1,s_2,u_2)}\\
&=\sum_{x_1,i_1,s_2} \PP^{\pi}_{\sigma,\tau}(x_1,i_1,s_2| d_1,j_1,u_2) \delta_{x_2(x_1,i_1,s_2,u_2)}.
\end{align*}
From the previous formula, we deduce
\begin{align*}
\PP^{\pi}_{\sigma,\tau}(x_1,i_1,s_2| d_1,j_1,u_2)&= \frac{\sum_{k_1} \pi(d_1)h_1(d_1)(x_1)x_1(k_1)\sigma(x_1,h_1(d_1))(i_1)\tau(d_1)(j_1)\hat{q}(k_1,i_1)(s_2,u_2)  }{\sum_{k'_1,x'_1,i'_1,s'_2} \pi(d_1)h_1(d_1)(x'_1)x'_1(k'_1)\sigma(x'_1,h_1(d_1))(i'_1)\tau(d_1)(j_1)\hat{q}(k'_1,i'_1)(s'_2,u_2)} \\
&= \frac{\sum_{k_1} f_1(c_1)(x_1)x_1(k_1)\sigma(x_1,f_1(c_1))(i_1) \hat{q}(k_1,i_1)(s_2,u_2)  }{\sum_{k'_1,x'_1,i'_1,s'_2} f_1(c_1)(x'_1) x'_1(k'_1)\sigma(x'_1,f_1(c_1))(i'_1)\hat{q}(k'_1,i'_1)(s'_2,u_2)}.
\end{align*}
Thus, $y_2$ does not depend on $\tau$ nor on $(j_1,d_1)$. Player $1$, knowing $i_1$, $c_1$ and $u_2$, can compute $y_2$, which proves that assumption $(A2)$ is satisfied.
Finally, the distribution of the random variable $y_2$ is equal to
\[ \eta_2= \sum_{c_1,i_1,u_2} \PP^{\pi}_{\sigma,\tau}(c_1,i_1,u_2) \delta_{y_2(c_1,i_1,u_2)}.\]
Since the function $\hat{q}$ does not depend on $j_1$, we deduce as above that $\eta_2$ does not depend on $\tau$ and therefore that $(A3)$ is satisfied.

\section{Discussion of the assumptions}\label{discuss}

In this section, we discuss several implications of the formal assumptions $(A1)$ ,$(A2)$, $(A3)$. 
At first, we show that it is necessary to introduce the notion of reduced strategy in order to exclude non relevant 
correlations between the players. Then, in order to answer to a suggestion made in \cite{R2012}, we show that the 
analysis cannot be made in terms of first order beliefs only, and that it is necessary to introduce second order
beliefs. Lemma  \ref{projection} shows that the value can be expressed as a function of second order beliefs. Then,
we prove that player $1$ can compute his first order beliefs, the second order beliefs of player $2$ and the distribution
of these beliefs without knowing the strategy of player $2$ as soon as he plays a Markovian strategy with respect to the 
beliefs at each stage. Finally, we give a weaker version of the theorem where the assumptions are formulated more directly
in terms of the data of the game.

\subsection{Necessity of reduced strategies.}
The introduction of reduced strategies for player $1$ is necessary in order to exclude non relevant correlations between players as shown in example \ref{reducednecessity} below. It will be shown in Lemma \ref{projection} and in the main Theorem that in our model, there is no loss in restricting player $1$ to reduced strategies.

\begin{example}\label{reducednecessity}
Let $K=\{k_1,k_2\}$, $I=\{T,B\}$, $C=\{a,b\}$, $D=\{\alpha,\beta\}$ and $J$ any finite set. The transition $q$ depends only on the action of player $1$ and  is described by the matrices
\begin{center}
\begin{tabular}{ccc} 
$\begin{matrix} T \\ B \end{matrix}$ & $\begin{pmatrix} \delta_{k_1} \\ \frac{1}{2} \delta_{k_1} + \frac{1}{2} \delta_{k_2} \end{pmatrix}$ & $\begin{pmatrix} \frac{1}{2} \delta_{k_1} + \frac{1}{2} \delta_{k_2} \\ \delta_{k_2} \end{pmatrix}.$ \\  & $k_1$ & $k_2$
\end{tabular} .
\end{center}
At each stage (including at the initial stage), the signals of the players are randomly chosen independently of the state variable with distribution
\begin{center}
\begin{tabular}{cc}
 & $\begin{matrix} \alpha & \beta \end{matrix}$ \\ $\begin{matrix} a \\ b \end{matrix}$ & $ \begin{pmatrix} 1/6 & 2/6 \\ 2/6 & 1/6 \end{pmatrix}$.
\end{tabular}
\end{center}
It is clear that signals do not contain any information on the state variable. However, assume that the initial state is $k_1$ and that player $1$ plays at the
first stage action $T$ if he receives the signal $a$ and action $B$ if he receives the signal $b$. The second order beliefs $y_2$ of player $2$ will differ if his initial private signal is equal to $\alpha$ or $\beta$. Since player $1$ is not able to compute the initial signal of player $2$, he is not able to compute the variable $y_2$ at the second stage. Nevertheless, when considering reduced strategies, signals can be omitted and player $1$ is able to compute
the beliefs of player $2$ which implies that $(A2)$ is satisfied.
\end{example}

\subsection{Second-order beliefs}

Renault \cite{R2012} conjectured that the pair of distributions of first order beliefs of both players could be
sufficient auxiliary variables. We present here an example showing the necessity to take into account second order
beliefs in the sense that there exist a game and two initial probabilities $\pi$ and $\pi'$ such that the law of first-order
beliefs are the same under $\pi$ and $\pi'$ while the values differ.

\begin{example}\label{secondorder}
We consider again the situation of example \ref{varaux}. Recall that $K=\{k_1,k_2\}$, that there are two public signals $U=\{u_1,u_2\}$ available to both players and three private signals $S=\{s_1, s_2,s_3 \}$ for player $1$. The set of signals of player $1$ is $C=U\times S$ and the set of signals of player $2$ is $D=U$. Let $\pi,\pi' \in \De(K\times S\times U)$ be defined by
\[ \begin{matrix}
 & \begin{matrix} u_1 & u_2  \end{matrix}& &\begin{matrix} u_1 & u_2  \end{matrix} \vspace{.2cm} \\ \begin{matrix} s_1 \\ s_2 \\ s_3 \end{matrix}  
 &\begin{pmatrix} \frac{8}{24} & 0 \\ \frac{1}{24} & \frac{3}{24} \\ 0 & 0 \\ \end{pmatrix}   & \begin{matrix} s_1 \\ s_2 \\ s_3 \end{matrix} 
 & \begin{pmatrix} 0 & 0 \\ \frac{1}{24} & \frac{3}{24} \\ \frac{2}{24} & \frac{6}{24} \\ \end{pmatrix} \vspace{.2cm}\\
 & k_1 & & k_2 
\end{matrix}.\]
and
\[ \begin{matrix}
 & \begin{matrix} u_1 & u_2  \end{matrix}& &\begin{matrix} u_1 & u_2  \end{matrix} \vspace{.2cm} \\
\begin{matrix} s_1 \\ s_2 \\ s_3 \end{matrix}  
 &\begin{pmatrix} \frac{6}{24} & \frac{2}{24} \\ \frac{3}{24} & \frac{1}{24} \\ 0 & 0 \\ \end{pmatrix} 
 & \begin{matrix} s_1 \\ s_2 \\ s_3 \end{matrix}  
 & \begin{pmatrix} 0 & 0 \\ \frac{3}{24} & \frac{1}{24} \\ 0 & \frac{8}{24} \\ \end{pmatrix} \vspace{.2cm}\\
 & k_1 & & k_2
\end{matrix}. \]
We will identify $\Delta(K)$ and $[0,1]$ as in example \ref{varaux}. The beliefs of Player $2$ about the state are the same in both cases and are equal to $\frac{1}{2}\delta_{\frac{3}{4}}+\frac{1}{2}\delta_{\frac{1}{4}}$. Similarly the beliefs of Player $1$ are $\frac{1}{3}\delta_{1}+\frac{1}{3}\delta_{\frac{1}{2}}+\frac{1}{3}\delta_{0}$ in both games. Moreover player $1$ observes the signal of player $2$, so that assumption $(A2)$ is satisfied. Thus, the laws of first-order beliefs are not sufficient to discriminate between $\pi$ and $\pi'$.
Let $\Gamma$ be the repeated game where $K=\{k_1,k_2\}$, $I=\{T,B\}$, $J=\{L,R\}$ and payoff $g$  given by
\[ \begin{matrix} \begin{pmatrix} 0 & 1 \\ 1 & 2 \\ \end{pmatrix} & \begin{pmatrix} 1 & \frac{1}{2} \\ 1 & 0 \\ \end{pmatrix} \\ k_1 & k_2 \end{matrix}.\]
The average payoff matrix with coefficients $(\frac{1}{2} ,\frac{1}{2})$ is $\begin{pmatrix} \frac{1}{2} & \frac{3}{4} \\ 1 & 1 \\ \end{pmatrix}.$

Let us prove that $v_1(\pi)$ and $v_1(\pi')$ are different. If player $1$ receives $s_3$, then his beliefs on the state is $0$ and thus Top is a weakly dominant action. If he receives $s_1$ or $s_2$, his beliefs is $1$ or $\frac{1}{2}$ and Bottom is a strictly dominant action. From the point of view of player $2$, playing Left if receiving $u_1$ and playing Right if receiving $u_2$ is a best reply to this strategy. Using these strategies,  we find that $v_1(\pi)=\frac{7}{8}$ and $v_1(\pi')=\frac{11}{12}.$
\end{example}

Let us prove that if assumptions $(A1a)$ and $(A2a)$ hold, then $v_\theta(\pi)$ depends only on the law of second order beliefs of player $2$. 
\begin{definition}
For all $\pi \in\De^*_f(K \times \NN \times \NN)$, define 
\[ \Phi(\pi) \triangleq  \CL_\pi\left( \CL_\pi \left( \CL_\pi ( k_1 | c_1 ) |d_1 \right)  \right). \]
\end{definition}

\begin{lemme}\label{projection}
Let $\pi,\pi' \in \De^*_f(K \times \NN \times \NN)$. If $\Phi(\pi)=\Phi(\pi')$, then $v_\theta(\pi)=v_\theta(\pi')$.
\end{lemme}
\begin{proof}
Let $(\sigma,\tau)$ be a pair of behavior strategies in $\Gamma_{\theta}(\pi)$. It is enough to show that $v_\theta(\pi)$ depends on $\pi$ only through $\eta_1 =  \Phi(\pi)$.
Recall that $x_1 := \CL_\pi ( k_1 | c_1 )$ and $y_1 :=\CL_\pi \left( \CL_\pi ( k_1 | c_1 ) |d_1 \right)$.
Note that $y_1$ is a function of $d_1$, and that $x_1$ is a function of $c_1$. Moreover, by assumption $(A2a)$, there exists a map $f_1^\pi : C' \rightarrow \De(\De(K))$ such that
\[ f_1^\pi (c_1)=y_1 \quad \pi\text{-almost surely}.\]
Let us construct a reduced version of the game $\Gamma_{\theta}(\pi)$ in which player $1$ and player $2$ are constrained to choose strategies that depend only on $c_1$ and $d_1$ through the variables $(x_1,y_1)$ and $y_1$ respectively,  and keeping the same payoff function. This game has a value since the sets of possible values of $(x_1,y_1)$ is finite and this value is exactly the value of $\Gamma_{\theta}(\tilde{\pi})$ where $\tilde{\pi}$ is the joint distribution of $(k_1,(x_1,y_1),y_1)$ seen as an element of $\Delta_f^*(K\times \NN \times \NN)$.

The sets of strategies in $\Gamma_{\theta}(\tilde{\pi})$ (denoted by $\Sigma'(\pi)$ and $\T'(\pi)$) can be seen as subsets of $\Sigma$ and $\T$ via the previous identification and we will prove that both games have the same value and that $v_{\theta}(\tilde{\pi})$ depends only on $\eta_1$.

Assume at first that $\tau \in \T'(\pi)$ and $\si\in  \Sigma$ and let $\mu$ denote the joint law of $(k_1,c_1,d_1,x_1,y_1)$ induced by $\pi$. By disintegration, we have
\begin{align*}
\gamma_\theta(\pi,\sigma,\tau)&=\int_{K\times \NN \times \De(K) \times \De(\De(K))} \gamma_\theta(k_1,\sigma(c_1),\tau(y_1))d\mu(k_1,c_1,x_1,y_1), \\
&=\int_{\NN \times \De(K) \times \De(\De(K))}  \left( \int_K  \gamma_\theta(k_1,\sigma(c_1),\tau(y_1)) d\CL_{\mu}(k_1|c_1,x_1,y_1) \right) d\mu(c_1,x_1,y_1),\\
&=\int_{\NN \times \De(K) \times \De(\De(K))} \int_{K}  \gamma_\theta(k_1,\sigma(c_1),\tau(y_1)) d\CL_{\mu}(k_1|c_1)  d\mu(c_1,x_1,y_1), \\
&=\int_{\NN \times \De(K) \times \De(\De(K))} \langle \gamma_\theta(\cdot,\sigma(c_1),\tau(y_1)), x_1 \rangle_{\RR^K} d\mu(c_1,x_1,y_1),
\end{align*}
where we used  that  $\CL_{\mu}(k_1 |c_1,x_1,y_1)=\CL_{\mu}(k_1 | c_1)$ since $(x_1,y_1)$ are $c_1$-measurable
and the notations $\gamma_\theta(.,\sigma(c_1),\tau(y_1))$ for $(\gamma_\theta(k,\sigma(c_1),\tau(y_1)))_{k\in K} \in \RR^K$
and $\langle \cdot, \cdot \rangle_{\RR^K}$ for the scalar product in $\RR^K$.
Taking the supremum over all strategies of player $1$, we obtain
\begin{align*}
\sup_{\sigma \in \Sigma} \gamma_\theta(\pi,\sigma,\tau)  &=\int_{\NN \times \De(K) \times \De(\De(K))} \sup_{\sigma(c_1)} \langle \gamma_\theta(\cdot,\sigma(c_1),\tau(y_1)), x_1 \rangle_{\RR^K} d\mu(c_1,x_1,y_1).
\end{align*}
The supremum  inside the integral is achieved by strategies depending only on $(x_1,y_1)$ since these variables are $c_1$ measurable. It means that there exists an optimal strategy in $\Sigma'(\pi)$, which proves
\[ \underset{\tau \in \T'(\pi)}{\inf} \; \underset{\si \in \Sigma }{\sup} \; \gamma_\theta(\pi,\sigma,\tau) = \underset{\tau \in \T'(\pi)}{\inf} \; \underset{\si \in \Sigma'(\pi) }{\sup} \; \gamma_\theta(\pi,\sigma,\tau).\]
Moreover the value of the reduced game depends only on $\eta_1$ since  taking the infimum over $\tau \in \T'(\pi)$,
\begin{align}
 \underset{\tau \in \T'(\pi)}{\inf} &\; \underset{\si \in \Sigma'(\pi)}{\sup}  \;\gamma(\pi,\sigma,\tau) \\
 &= \int_{\De(\De(K))} \left[ \inf_{\tau(y_1)} \int_{\De(K)} \left(\sup_{\sigma(x_1,y_1)} \langle \gamma_\theta(\cdot,\sigma(x_1,y_1),\tau(y_1)), x_1 \rangle_{\RR^K} \right) d\CL_{\mu}(x_1\mid y_1)\right] d\mu (y_1),
 \end{align}
which depends only on the law of $y_1$, since  $y_1=\CL_{\mu}(x_1 \mid d_1)= \CL_{\mu}(x_1 \mid y_1)$.

Let us prove a dual equality starting with $\si \in \Sigma'(\pi)$ and $\tau \in \T$:
\begin{align*}
\gamma_\theta(\pi,\sigma,\tau)&=\int_{K\times \NN \times \NN \times \De(K) \times \De(\De(K))} \gamma_\theta(k_1,\sigma(x_1,y_1),\tau(d_1))d\mu(k_1,c_1,d_1,x_1,y_1), \\
&=\int_{\NN \times \NN \times \De(K)\times \De(\De(K))}  \left( \int_K  \gamma_\theta(k_1,\sigma(x_1,y_1),\tau(d_1)) d\CL_{\mu}(k_1|c_1,d_1,x_1,y_1) \right) d\mu(c_1,d_1,x_1,y_1)\\
&=  \int_{\NN \times \NN \times \De(K)\times  \De(\De(K))}  \langle \gamma_\theta(\cdot,\sigma(x_1,y_1),\tau(d_1)), x_1 \rangle_{\RR^K} d\mu(c_1,d_1,x_1,y_1)\\
&= \int_{\NN \times \De(K)\times \De(\De(K))}  \langle \gamma_\theta(\cdot,\sigma(x_1,y_1),\tau(d_1)), x_1 \rangle_{\RR^K} d\mu(d_1,x_1,y_1)\\
&= \int_{\NN\times \De(\De(K))} \left(\int_{\De(K)} \langle \gamma_\theta(\cdot,\sigma(x_1,y_1),\tau(d_1)), x_1 \rangle_{\RR^K} d\CL_{\mu}(x_1 \mid d_1,y_1) \right) d\mu(d_1,y_1). \end{align*}
For the second equality, we used that $\CL_{\mu}(k_1 \mid c_1,d_1,x_1,y_1)=\CL_{\mu}(k_1 \mid c_1,d_1)= \CL_{\mu}(k_1 \mid c_1)=x_1$ which follows from the fact that $(x_1,y_1)$ is $c_1$-measurable and assumption $(A1)$. Taking the infimum over all $\tau \in \T$, it follows that
\begin{align}
 \inf_{\tau \in \Tau} \gamma(\pi,\sigma,\tau)&= \int_{\NN\times \De(\De(K))}  \inf_{\tau(d_1)}\left( \int_{\De(K)} \langle \gamma_\theta(\cdot\sigma(x_1,y_1),\tau(d_1)), x_1 \rangle_{\RR^K} d\CL_{\mu}(x_1 \mid d_1) \right) d\mu(d_1,y_1).
\end{align}
The infimum inside the integral is achieved for strategies depending only on $y_1= \CL_{\mu}(x_1\mid d_1)$ since $y_1$ is $d_1$-measurable. We proved that
\[ \underset{\si \in \Sigma'(\pi)}{\sup} \; \underset{\tau \in \T}{\inf} \;  \gamma_\theta(\pi,\sigma,\tau) = \underset{\si \in \Sigma'(\pi) }{\sup} \;\underset{\tau \in \T'(\pi)}{\inf} \; \gamma_\theta(\pi,\sigma,\tau). \]
Finally, using that $\T'(\pi) \subset \T$ and $\Sigma'(\pi) \subset \Sigma$, it follows that
\[ v_{\theta}(\pi) =\underset{\si \in \Sigma}{\sup} \; \underset{\tau \in \T}{\inf} \;  \gamma_\theta(\pi,\sigma,\tau) \geq \underset{\si \in \Sigma'(\pi)}{\sup} \; \underset{\tau \in \T}{\inf} \;  \gamma_\theta(\pi,\sigma,\tau) = v_{\theta} (\tilde{\pi}),\]
\[ v_{\theta}(\pi) = \underset{\tau \in \T}{\inf} \; \underset{\si \in \Sigma}{\sup} \; \gamma_\theta(\pi,\sigma,\tau) \leq  \underset{\tau \in \T'(\pi)}{\inf} \; \underset{\si \in \Sigma}{\sup} \; \gamma_\theta(\pi,\sigma,\tau) = v_{\theta} (\tilde{\pi}),\]
which proves the equality. Since $v_{\theta}(\tilde{\pi})$ depends only on $\eta_1$, the proof is complete.
\end{proof}

\subsection{Player $1$ can compute his beliefs without knowing player $2$'s strategy.}

Assumption $(A1)$ can be reformulated as a couple of assumptions $(A1a)$ and $(A1b)$ which are expressed in terms of $\pi$, i.e. the initial information, and $q$, i.e. the evolution of the information structure for stages $m\geq2$ respectively.
\begin{itemize}
\item[(A1a)] The probability $\pi \in \Delta( K \times C' \times D')$ is such that
\begin{equation*} 
\forall (k,c',d') \in K\times C' \times D', \qquad \pi(c')\pi(k,c',d')= \pi(k,c')\pi(c',d')
\end{equation*}
\item[(A1b)] There exists a map $F$ from $\Delta(K) \times I \times C$ to $\Delta(K)$ such that  
\end{itemize}
\begin{equation*} 
\forall (p,i,j,c,d,k) \in \De(K)\times I \times J \times C \times D\times K, \; q(p,i,j)[k,c,d]= F(p,i,c)[k] \sum_{k' \in K} q(p,i,j)[k',c,d].
\end{equation*}

Note that $(A1a)$ is equivalent to $\pi(k|c',d')=\pi(k|c')$ for any $(k,c',d')$ such that $\pi(c',d')>0$. Similarly $(A1b)$ could be written in terms of conditional probabilities, though we shall distinguish events with probability $0$. In addition, it highlights the first important consequence of assumption $(A1)$: player $1$ can compute his beliefs about the state variable (i.e. the conditional distribution in the right-hand-side of $(A1)$) without knowing the strategy, nor the signals, of his opponent.

\begin{proposition}
Assuming $(A1)$, player $1$ can compute $x_n$ for each $n \geq 1$ without knowing the strategy of player $2$.
\end{proposition}

The proof of the Proposition follows directly from the following Lemma.

\begin{lemme}\label{eqbelief}
Assumptions $(A1)$ and $(A1a+A1b)$ are equivalent. Furthermore, the map $F$ from $\Delta(K) \times I \times C$ to $\Delta(K)$ defined in $(A1b)$ is such that for all $n \geq 2$ and  for all strategy profile $(\si,\tau)$
\[ x_{n} = F(x_{n-1},i_{n-1},c_{n-1}), \quad \PP^{\pi}_{\si\tau} \text{-almost surely}.\]
\end{lemme}

\begin{proof}
Using the definition of conditional independence, assumption $(A1)$ at stage $1$ is equivalent to $(A1a)$. It remains to prove that $(A1)$ for $n\geq 2$ implies $(A1b)$ and the converse. Assume that $\pi$ fulfils $(A1a)$ and let $(\si_1,\tau_1)\in\De(I)^{C'}\times \De(J)^{D'}$ be strategies with full
support. By construction, we have
\[\CL_{\PP^\pi_{\si\tau}}(k_{2},c_2,d_2 \mid k_{1},c_1,i_1,d_1,j_1)= q(k_{1},i_1,j_1)\in \De(K\times C\times D). \]
It follows, using the tower property of conditional expectation and $(A1a)$ that
\[ \CL_{\PP^\pi_{\si\tau}}(k_{2},c_2,d_2 \mid c_1,i_1,d_1,j_1)= q(x_1,i_1,j_1),\]
where, by definition, $x_1$ can be written as a function of $c_1$.
On one hand, one obtains by disintegration
\[ \PP^{\pi}_{\sigma \tau}(k_2=k \mid c_2,d_2,c_1,i_1,d_1,j_1)(\sum_{\widetilde{k} \in K } q(x_1,i_1,j_1)[\widetilde{k},c_2,d_2]) = q(x_1,i_1,j_1)[k,c_2,d_2].  \]
On the other hand, the conditional law $\CL_{\PP^\pi_{\si\tau}}(k_2\mid c_1,i_1,c_2)$ is characterized by  the following expression
\begin{align*}
\PP^{\pi}_{\sigma \tau} &(k_2=k\mid c_1,i_1,c_2)(\sum_{\widetilde{k},\widetilde{d}_1,\widetilde{d}_2,\widetilde{j}_1} \pi(c_1,\widetilde{d}_1)\tau_1(\widetilde{d}_1)[\widetilde{j}_1]q(x_1(c_1),i_1,\widetilde{j}_1)[\widetilde{k},c_2,\widetilde{d}_2] )  \\
&=  \sum_{\widetilde{d}_1,\widetilde{d}_2,\widetilde{j}_1}
\pi(c_1,\widetilde{d}_1)\tau_1(\widetilde{d_1})[\widetilde{j}_1]q(x_1(c_1),i_1,\widetilde{j}_1)[k,c_2,\widetilde{d}_2].
\end{align*}
Assumption (A1) for $n=2$ implies that these two conditional
probabilities are equal, which in turn implies
\begin{align}\label{equivA1}
\frac{q(x_1,i_1,j_1)[k,c_2,d_2]}{\sum_{\widetilde{k} \in K }
q(x_1,i_1,j_1)[\widetilde{k},c_2,d_2]} =
\frac{\sum_{\widetilde{d}_1,\widetilde{d}_2,\widetilde{j}_1}
\pi(c_1,\widetilde{d}_1)\tau_1(\widetilde{d}_1)[\widetilde{j}_1]q(x_1(c_1),i_1,\widetilde{j}_1)[k,c_2,\widetilde{d}_2]}{\sum_{\widetilde{k},\widetilde{d}_1,\widetilde{d}_2,\widetilde{j}_1}
\pi(c_1,\widetilde{d}_1)\tau_1(\widetilde{d}_1)[\widetilde{j}_1]q(x_1(c_1),i_1,\widetilde{j}_1)[\widetilde{k},c_2,\widetilde{d}_2]}
\end{align}
whenever the left-hand side is well-defined. Since $\tau_1$ has full support, this implies that the right-hand side is also well-defined in this case and does not depend on $d_1,j_1,d_2$. Moreover, for all $p \in \Delta(K)$, we can choose an initial distribution $\pi$ such that $\pi( x_1 =p) > 0$. It follows that there exists a function $F$ such that 
\[ F(p,i,c)[k]= \frac{q(p,i,j)[k,c,d]}{\sum_{k' \in K } q(p,i,j)[k',c,d]},\]
whenever the right hand side is well-defined for some $(j,d)$ and extended by $1/|K|$ (say) otherwise.

For the converse assertion, we already mentioned that $(A1a)$ implies $(A1)$ for $n=1$. We are therefore allowed to write the following formula for the conditional laws,
\begin{align}\label{equ_cond_law}
\PP(k_2=k \mid c_2,d_2,c_1,i_1,d_1,j_1) &=\frac{q(x_1,i_1,j_1)[k,c_2,d_2]}{\sum_{\widetilde{k} \in K } q(x_1,i_1,j_1)[\widetilde{k},c_2,d_2]}.
\end{align}
It follows therefore that
\[\PP(k_2=k \mid c_2,d_2,c_1,i_1,d_1,j_1) = F(x_{1},i_{1},c_{2}), \quad \PP^{\pi}_{\si\tau} \text{-almost surely},\]
and since the right-hand-side is measurable with respect to the history of player $1$, we have the equality
\begin{align*}
\PP(k_2=k \mid c_2,c_1,i_1) & = \EE \big[ \PP(k_2=k \mid c_2,d_2,c_1,i_1,d_1,j_1) | c_1,i_1,c_2  \big], \\
                           & = F(x_{1},i_{1},c_{2}), \\
                           & = \PP(k_2=k \mid c_2,d_2,c_1,i_1,d_1,j_1).
\end{align*}
which proves $(A1)$ and our last assertion for $n=2$. Finally the distribution of $(k_2,(c_1,i_1,c_2),(d_1,j_1,d_2))$, seen as an element of $\De_f(K\times \NN \times \NN)$, fulfils $(A1a)$. Applying exactly the same argument with these new initial signals allows us therefore to conclude by induction on $n$.
\end{proof}

\subsection{Player $1$ can compute the beliefs of player $2$.}

The assumptions $(A1)$ and $(A2)$ are independent, as shown in  example \ref{2etpas1} below. However, $(A2)$ 
really makes sense only when player $1$ is better informed. 

\begin{example}\label{2etpas1}
Let $\Gamma=(K,I,J,C,D,q,g)$  be  such that player $1$ is in the dark and player $2$ is perfectly informed: $K=\{\alpha,\beta\}$, $I$ and $J$ are finite, $C$ is a singleton $\{c\}$ and $D=K$. The payoff mapping is anything and the state is randomly chosen at each stage with probability $(1/2,1/2)$. Player $1$ observes nothing and player $2$ learns the state.
It is clear that player $1$'s signal is less accurate than player $2$'s, so that assumption $(A1)$ is not satisfied. On the other hand, $(A2)$ is satisfied since player $1$ knows the beliefs of player $2$ about himself which is $(\frac{1}{2},\frac{1}{2})$ whatever are the signals.
\end{example}

Under the assumptions $(A1)$ and $(A2)$, if player $1$ plays a reduced strategy, he can compute $y_2$, the belief of
player $2$ about his own belief on the state, without knowing the strategy of player $2$.

\begin{lemme}\label{f2}
Assume $(A1b)$ and $(A2b)$, and let $\pi \in \Delta^*_f(K \times \NN \times \NN)$. Then, for all $\sigma \in \Sigma'(\pi)$, there exists a map $f_2=f_2^{\pi,\sigma}$ such that for all $\tau \in \Tau$
\[ y_2= f_2(h^{I}_2), \; \PP^{\pi}_{\sigma \tau}-\text{almost surely}. \] 
\end{lemme}

\begin{proof}
It is sufficient to prove that the map $f_2$ appearing in $(A2b)$ does not depend on $\tau$. Note that since we assumed $(A1)$,
we have $x_2 = F(x_1(c_1),i_1,c_2)$ almost surely, where $F$ is defined in $(A1b)$. Moreover, the conditional probability 
\begin{align*}
 \PP^{\pi}_{\sigma \tau} (c_1=\tilde{c}_1,i_1=\tilde{i}_1,&c_2=\tilde{c}_2 | d_1,j_1,d_2) \\
 &= \frac{\pi(\tilde{c}_1,d_1) \sigma(x_1(\tilde{c_1}),y_1(d_1))(\tilde{i}_1) \tau(d_1)(j_1) q(x_1(\tilde{c}_1,\tilde{i}_1,j_1)(\tilde{c}_2,d_2) }{\sum_{c_1',i_1',c_2'} \pi(x',d_1) \sigma(x_1(c_1'),y_1(d_1))(i_1') \tau(d_1)(j_1) q(x_1(c_1'),i_1',j_1)(c_2',d_2)} \\
 &= \frac{\pi(\tilde{c}_1,d_1) \sigma(x_1(\tilde{c_1}),y_1(d_1))(\tilde{i}_1)  q(x_1(\tilde{c}_1,\tilde{i}_1,j_1)(\tilde{c}_2,d_2) }{\sum_{c_1',i_1',c_2'} \pi(x',d_1) \sigma(x_1(c_1'),y_1(d_1))(i_1')  q(x_1(c_1'),i_1',j_1)(c_2',d_2)} 
 \end{align*}
 does not depend on $\tau$. There exists therefore a map $y_2(d_1,j_1,d_2)$ which does not depend on $\tau$, defined by the above expression everywhere it makes sense and arbitrarily elsewhere. Let $\tau^*$ be a strategy with full support. Using $(A2b)$, there exists a map $f_2^{\pi,\sigma,\tau^*}$ such that 
 \[ y_2(d_1,j_1,d_2) = f_2^{\pi,\sigma,\tau^*} (c_1,i_1,c_2), \; \PP^{\pi}_{\sigma \tau^*}-\text{almost surely}. \]
The previous computation shows that the conditional law of $(c_1,i_1,c_2)$ given $(d_1,j_1,d_2)$ does not depend on $\tau$.  Therefore, if the event
 \[\{d_1=\tilde{d}_1,j_1=\tilde{j}_1,d_2=\tilde{d}_2, c_1=\tilde{c}_1,i_1=\tilde{i}_1,c_2=\tilde{c}_2\} \]
has positive probability under  $\PP^{\pi}_{\sigma \tau}$, it also has positive probability under $\PP^{\pi}_{\sigma \tau^*}$. We deduce that $f_2^{\pi,\sigma,\tau}=f_2^{\pi,\sigma,\tau^*}$, $\PP^{\pi}_{\sigma \tau}$-almost surely for all $\tau$, which concludes the proof.   
\end{proof}

Let us now prove that player $1$ is able to play a strategy which is Markovian with respect to the beliefs. The idea is to prove by induction that if player $1$ plays a strategy which depends at stage $n-1$ only on $(x_{n-1},y_{n-1})$, then he can compute the variables $(x_n,y_n)$ at stage $n$ and play at stage $n$ a strategy which depends only on $(x_n,y_n)$, etc...
Formally, we have the following.

\begin{lemme}\label{markovstrat}
For all $\pi \in \Delta^*_f(K \times \NN \times \NN)$, and for any sequence of $\Delta(I)$-valued measurable functions $\psi_1,\psi_2,...$ defined on $\Delta(K)\times \Delta_f(\Delta(K)$, there exists a strategy $\sigma$ such that for all $\tau$ and for all $n$
\[ \sigma(h^{I}_n)= \psi_n(x_n,y_n), \quad \PP^{\pi}_{\sigma \tau}-\text{almost surely}.\]
\end{lemme}

\begin{proof}
We will prove the result by induction. It is obviously true for $n=1$ due to the definition of $\Delta^*_f(K \times \NN \times \NN)$. 
For $n=2$, due to the Lemmas \ref{eqbelief}  and \ref{f2}, player $1$ can compute $x_2$ and $y_2$ as a function of $h^{I}_2$ independently of the chosen strategy of player $2$. However, to prove the property for $n \geq 3$, we cannot rely on the same argument. It would be tempting to say that the distribution of $(k_2, h^{I}_2,h^{II}_2)$ belongs to $\Delta^*_f(K \times \NN \times \NN)$ and to apply the preceding argument when starting from this new initial distribution. But this would be wrong since this distribution may depend on $\tau$. To overcome this problem, it is sufficient to prove that the map $f_2$ appearing in $(A2b)$ and the distribution $\eta_2$ appearing in $(A3)$ depend on $\pi$ only through $\Phi(\pi)$.  Indeed, in this case, reasoning by induction, player $1$ can compute $\eta_n$ as a function of $\eta_{n-1}$ and his new signals, $x_n$  using Lemma \ref{eqbelief}, and $y_n$ using the map given by $(A2b)$ which will depend only on $\eta_{n-1}$ and his own strategy.   

Let us prove this assertion. Let $\pi \in \Delta^*_f(K\times \NN \times \NN)$, and $\sigma \in \Sigma'(\pi)$ be a reduced strategy, which implies that there exists a map $\psi : K\times \Delta(K) \rightarrow \Delta(I)$ such that $\pi$-almost surely $\sigma_1(c_1)=\psi(x_1,y_1)$.  Assumption $(A3)$ implies that $\eta_2$ is a function of the initial distribution $\pi$ and $\sigma_1$ only. We denote it by $\eta_2(\pi,\sigma)$. We now prove that $\eta_2(\pi,\sigma)$ and the map $f_2^{\pi,\sigma}$ appearing in $(A2b)$ depend only on the projection of $\pi$, $\eta_1=\Phi(\pi)$, and on the map $\psi$.

At first, given $\eta_1 \in \DDDK$, we can construct a canonical probability $\overline{\pi}$ with finite support on $K\times \Delta(K) \times \Delta_f(\Delta(K))$ defined by   $\overline{\pi}(k,p,z)=p^kz(p)\eta(z)$. Applying $(A3)$ and $(A2)$ in the game $\Gamma(\overline{\pi})$ if player $1$ plays $\overline{\sigma}_1=\psi$,  there exists a distribution $\eta_2(\overline{\pi},\psi)$ and a map $f_2^{\overline{\pi},\psi} : K \times \Delta(K)$ such that $y_2=f_2^{\overline{\pi},\psi}(x_1,y_1)$ almost surely and $y_2$ has law $\eta_2(\overline{\pi},\psi)$ for all $\tau$. Recall that $\pi$ is such that $\Phi(\pi)=\eta_1$ and  that $\si_1$  is such that $\sigma_1(c_1)=\psi(x_1,y_1)$. Note also that $d_1$ and $x_1$ are conditionally independent given $y_1$ under $\pi$. Therefore, for any $\tau_1$, the joint law of $(x_1,y_1,i_1,j_1,c_2,d_2)$  is the same under $\PP^{\pi}_{\si_1,\tau_1}$ and under the probability $\PP^{\overline{\pi}}_{\psi, \tau'_1}$ where $\tau'_1$ is defined as follows: choose $d_1$ using some exogenous lottery 
such that the conditional law of $d_1$ given $y_1$ is the same as under $\pi$ and then play $\tau_1(d_1)$. We deduce that $\eta_2(\pi,\sigma_1)=\eta_2(\overline{\pi},\psi)$ and $y_2=f_2^{\overline{\pi},\psi}(x_1,y_1)$ under the probability $\PP^{\pi}_{\si_1,\tau_1}$ which concludes the proof.
\end{proof}

\subsection{A stronger version of the theorem}

To conclude this section, let us state a couple of stronger assumptions, which are expressed in terms of the data of the game more directly: Player $1$ can deduce exactly the signal received by player $2$ and player $2$ can not influence the joint law of $(x_2,d_2)$. 
\begin{definition}
For all, $x,i,j \in \Delta(K)\times I \times J$, let $q_{C\times D}(x,i,j)$ denote the marginal distribution on $C\times D$ induced by $q(x,i,j)$, i.e. $q_{C\times D}(x,i,j)(c,d)=\sum_{k,\tilde{k}}x(k)q(k,i,j)(\tilde{k},c,d)$.
\p
Let also $H_{x,i}$ the map defined on $C \times D$ by
\[ H_{x,i}(c,d)= (F(x,i,c),d) \in \De(K)\times D.\]
\end{definition}
With these notations, we can define a set of assumptions on the marginal of $q$. The assumptions $(A1)$, $(A2a)$ are unchanged and we define $(A'2b)$ and $(A'3)$.
\begin{itemize}
\item[$(A'2b)$] Player $1$ knows the signal of player $2$ i.e. there exists a map $h:C \rightarrow D$ such that for all $(k,i,j)\in K\times I \times J$,
$\sum_{c\in C} q(k,i,j)[c,h(c)]=1$.
\item[$(A'3)$]  The image probability $\phi(x,i)$ of $q_{C\times D}(x,i,j)$ by the map $H_{x,i}$ does not depend on $j$.
\end{itemize}

\begin{corollaire}
Let $\Ga$ be such that assumptions $(A1), (A2a), (A'2b)$ and $(A'3)$ are true. Then:
\[ \text{For all} \;\pi \in \De^*_f(K\times \NN \times \NN), \quad \Ga(\pi) \; \text{has a uniform value.} \]
\end{corollaire}
The proof of this corollary follows directly from the next Lemma.
\begin{lemme}
If $A1$ and $A2a$ hold, then $A'2b$ and $A'3$ imply $A2b$ and $A3$.
\end{lemme}
\begin{proof}
It follows from the definitions and from Lemma \ref{eqbelief} that
\begin{align*}
 \CL_{\PP^\pi_{\si\tau}} (x_2,d_2 \mid c_1,d_1,i_1,j_1)&= \CL_{\PP^\pi_{\si\tau}} (F(x_1,i_1,c_2),d_2 \mid c_1,d_1,i_1,j_1)\\
 &= \CL_{\PP^\pi_{\si\tau}} (H_{x_1,i_1}(c_2,d_2) \mid c_1,d_1,i_1,j_1)= \phi(x_1,i_1),
\end{align*}
since $(x_1,i_1)$ is measurable with respect to $(c_1,d_1,i_1,j_1)$ and $\phi(x_1,i_1)$ is the image probability of $q_{C\times D}(x_1,i_1,j_1)$ by the map $H_{x_1,i_1}$. Therefore, the conditional law of the pair $(x_2,d_2)$ does not depend on the strategy of player $2$. Precisely, we have
\[ \CL_{\PP^\pi_{\si\tau}} (x_2,d_2 \mid d_1,j_1)= \EE^\pi_{\si\tau}[\phi(x_1,i_1) \mid d_1,j_1]. \]
Since $j_1$ and $(x_1,i_1)$ are conditionally independent given $d_1$ it follows that
\begin{align*}\label{eq_red}
\CL_{\PP^\pi_{\si\tau}} (x_2,d_2 \mid d_1,j_1) &= \EE^\pi_{\si\tau}[\phi(x_1,i_1) \mid d_1].
\end{align*}
The right hand side does not depend on $\tau$, so $\CL_{\PP^\pi_{\si\tau}} (x_2,d_2 \mid d_1,j_1)$ does not depend on $\tau$ and $j_1$, and the same is true for the (unconditional) law of $(x_2,y_2)$. As a consequence, the law of $y_2$ (denoted $\eta_2$) does not depend on $\tau$ which proves $(A3)$. It remains to prove that player $1$ can compute the auxiliary random variable $y_2$.  Using $(A2a)$ and that $\si$ is reduced, $i_1$ can be written as a measurable function of $(x_1,y_1)$ and of an independent random variable $u$ uniformly distributed on $[0,1]$. Recall that the conditional law of $x_1$ given $d_1$ is $y_1$, so that
\begin{align*}
\CL_{\PP^\pi_{\si\tau}} (x_2,d_2 \mid d_1,j_1) & =\EE^\pi_{\si\tau}[\phi(x_1,i_1(x_1,y_1,u)) \mid d_1]\\
& = \int_{\De(K)\times [0,1]} \phi(x,i_1(x,y_1,u)dy_1[x]du.
\end{align*}
Player $1$ can compute the conditional law of $(x_2,d_2)$ given $(d_1,j_1)$ since it depends only on $(y_1,\si)$. Moreover by assumption $(A'2b)$, he can deduce $d_1$ from his initial signal $c_1$, so he is able to compute $y_2$ which proves $(A2b)$.
\end{proof}

\section{Proof of Theorem \ref{main1}.}\label{proof}

The proof is divided into three steps. First, using Lemma \ref{projection}, we define a value function $\hat{v}$ on $\De_f(\De_f(\De(K)))$ and prove that it is concave and Lipschitz. Secondly, we introduce an auxiliary game $\G$ on $\Delta_f(\Delta(K))$ and check it satisfies some (slightly) weakened assumptions needed to apply a Theorem of Renault \cite{R2012}. This implies the existence of a uniform value in the auxiliary game. Finally we show that both players can guarantee this value in the original game: player $2$ by playing by blocks and player $1$ by using optimal Markovian strategies in the auxiliary game.

\subsection{The canonical value function $\hat{v}_\theta$}

In view of Lemma \ref{projection}, it is appropriate to work directly on the set $\De_f(\De_f(\De(K)))$, i.e. for any $\pi,\pi'$ such that $\Phi(\pi)=\Phi(\pi')$ the value of the game is the same. At first, given $\eta\in \De_f(\De_f(\De(K)))$, there is a canonical way to build a distribution $\pi$ such that $ \phi(\pi)=\eta$. 
\begin{definition}
Let $\Gamma=(K,I,J,C,D)$ be a repeated game. For any $\eta \in \De_f(\De_f(\De(K)))$, we define $D'={\rm supp}(\eta) \subset \Delta_f(\Delta(K))$ and $C'=D' \times \left(\cup_{z \in {\rm supp}(\eta)} {\rm supp}(z) \right)$. By definition of $\eta$, these sets are finite and we can define $\pi \in \Delta^*_f(K\times C' \times D')$ by
\[ \forall (k,p,z)\in K \times \Delta(K)\times \Delta_f(\Delta(K)), \, \pi(k,(p,z),z)=\eta(z)z(p)p(k). \]
To canonical game $\Ga(\pi)$ will be denoted $\widehat{\Ga}(\eta)$, and its value $\hat{v}_{\theta}(\eta)$. If $\eta=\de_z$ for some $z\in \De_f(\DeK)$, we will use the shorter notations $\widehat{\Ga}(z)=\widehat{\Ga}(\delta_z)$ and $\hat{v}_{\theta}(z)$ for the value.
\end{definition}

Informally, the game $\widehat{\Ga}(\eta)$ proceeds as follows:  $\eta$ is common knowledge, player $2$ is informed about the realization $z$ of a random variable of law $\eta$ (player $2$ learns his beliefs). Then player $1$ is informed about $z$ (his opponent's beliefs) and about the realization $p$ of a random variable of law $z$ (his own beliefs). The state variable is finally selected according to $p$, but none of the players observe it. 
If $\eta=\de_z$, for some $z\in \De_f(\DeK)$, then the set of initial signals for player $2$ is reduced to a singleton. In this case, player $1$ receives a partial information about the state, whereas player $2$ only knows the joint distribution over the state and player $1$'s signal.
Using these notations,  Lemma \ref{projection} implies that if $\pi,\pi' \in \De_f^*(K\times \NN\times \NN)$ are such that $\Phi(\pi)=\Phi(\pi')$, we have that $v_{\theta}(\pi)=v_{\theta}(\pi')=\hat{v}_{\theta}(\Phi(\pi))$.
\p
In order to study the regularity of the canonical value function, let us recall some properties of the Wasserstein distance 
\p
Let $(Z, {\rm d})$ be a compact metric space and $Lip_1(Z)$ the set of 1-Lipschitz functions on $Z$. The function
\[ \mathbf{d}: \Delta(Z) \times \Delta(Z) : (\mu,\nu) \rightarrow \underset{f \in Lip_1(Z)}{sup} \int_Z fd\mu - \int_Z f d\nu \]
is a distance on $\Delta(Z)$ which makes $\Delta(Z)$ compact. Moreover, for all $\mu,\nu \in \De(Z)$
\[ \mathbf{d}(\mu,\nu) = \underset{\pi \in \mathcal{P}(\mu,\nu)}{min} \int_{Z\times Z} |y-x|d\pi(x,y),\]
where $\mathcal{P}(\mu,\nu)$ is the set of probabilities on $Z\times Z$ having for marginals $\mu$ and $\nu$ (see e.g. \cite{villani}).

If $f$ is a bounded measurable function on $Z$, define $\tilde{f}:\Delta(Z) \rightarrow \RR$ by $\tilde{f}(\mu)=\int_Z f d\mu$. Then
\[\tilde{f} \in Lip_1(\Delta(Z),\mathbf{d}) \Leftrightarrow f \in Lip_1(Z).\]

In the following, $\De(K)$ is endowed with the $\ell_1$-norm induced by $\RR^K$ and $\De(\De(K))$ is endowed with the Wasserstein metric $\mathbf{d}$ induced by the metric space $(\De(K),\ell_1)$.

\begin{lemme}\label{canlip}
Let $\eta \in \De_f(\De_f(\De(K)))$ and  $z \in  \De_f(\De(K))$. Then $\hat{v}_{\theta}(\eta)$ is linear on $\Delta_f(\Delta_f(\Delta(K)))$ and the mapping on $\De(\De(K))$, $\hat{v}_{\theta}(z)$ is $1$-Lipschitz for the Wasserstein metric $\mathbf{d}$.
\end{lemme}

\begin{proof}
The first assertion is immediate since by definition both players learn the realization of $\eta$. Let $z$,$z' \in \De_f(\De(K))$. By definition of the Wasserstein distance, there exists $\mu \in \De(\De(K)\times \De(K))$ such that the first marginal is $z$, the second is $z'$ and
\[\mathbf{d}(z,z')=\int_{\De(K)\times \De(K)} \|p-p'\|_1 d\mu(p,p').\]
We denote by  $\CL_{\mu}(p|p')$ the conditional law of $p$ given $p'$.

Let $\sigma \in \Sigma$ be a behavior strategy for player $1$ in the game $\widehat{\Ga}(z)$. As in Section \ref{miss},  $\si(p)$ denotes the strategy of player $1$ conditionally on the  signal $p$. Let us construct a general strategy for P1 as follows. Let $(\Omega,\PP)=([0,1],dx)$ be the auxiliary probability space\footnote{Using a continuum of alternatives is clearly unnecessary but allows to simplify the proof.} that will be used as a ``tossing coin''. The classical representation result of Blackwell-Dubins (see \cite{blackwelldubins}) asserts that there exists a jointly Borel-measurable map $\phi :\Omega\times \Delta(\Delta(K)) \mapsto \De(K)$ such that for all $\nu \in \De(\De(K))$, $\phi(\cdot,\nu)$ is a $\nu$-distributed random variable. Therefore, the map $\sigma'(\omega,p')=\sigma(\phi(\omega,\mu(p|p')))$ defines a general strategy which is equivalent to a behavior strategy by Kuhn's theorem. It follows that
\begin{align*}
 \gamma_{\theta}(z',\sigma',\tau) & =\int_{\De(K)\times \Omega} \gamma_\theta(p',\sigma'(\omega,p'),\tau)dz'(p')\otimes d\PP(\omega), \\
                         & =\int_{\De(K)} \left(\int_\Omega \gamma_\theta(p',\sigma(\phi(\omega,\mu(p|p'))), \tau)d\PP(\omega) \right)dz'(p'), \\
                         & =\int_{\De(K)} \left(\int_{\De(K)} \gamma_\theta(p',\sigma(p), \tau)d\CL_{\mu}(p|p') \right)d\mu(p'), \\
                         & =\int_{\De(K)\times \De(K)} \gamma_\theta(p',\sigma(p), \tau)d\mu(p,p'), \\
\end{align*}
where the last equality follows from $dz'(p')=d\mu(p')$. Recall that by assumption $g$ takes values in $[0,1]$. Consequently, $\ga(p,\si(p),\tau)\in[0,1]$, $\forall p,\si,\tau$. Hence 
\begin{align*}
 |\gamma_\theta(z',\si',\tau)-\gamma_\theta(z,\si,\tau)| & \leq \int_{\De(K)\times \De(K)} |\gamma_\theta(p,\sigma(p),\tau)-\gamma_\theta(p',\sigma(p), \tau)|d\mu(p,p'), \\                      & \leq \int_{\De(K)\times \De(K)} \|p-p'\|_1 d\mu(p,p'), \\
                                           & = \mathbf{d}(z,z').
\end{align*}
It follows that $|\hat{v}_{\theta}(z)-\hat{v}_{\theta}(z')|\leq \mathbf{d}(z,z')$, for any $z,z' \in \De_f(\DeK)$.
\end{proof}

Note that usually, the underlying space is $\Delta(K)$ with discrete metric on $K$ and, in order to prove that the value is $1$-Lipschitz, we can use the same strategy in $\Gamma(z)$ in $\Gamma(z')$. Here, we cannot use directly $\sigma$. The state space is $\Delta_f(\Delta(K))$ with the norm $1$ on $\Delta(K)$, and  two states may be close while having disjoint supports. Therefore an optimal strategy $\sigma$ in $\Gamma(z)$ may have no sense in $z'$. The idea behind the above proof is to construct, given $\sigma$ in $\Gamma(z)$, a strategy $\sigma'$ in $\Gamma(z')$ which behaves in $z'$ like $\sigma$ in $z$.

\begin{example}
Assume that $K=\{k_1,k_2\}$ and let  $z=\delta_{\frac{1}{2}}$ and $z'=\frac{1}{2} \delta_{\frac{1}{2}-\epsilon}+\frac{1}{2} \delta_{\frac{1}{2}+\epsilon}$ be two initial distributions in $\Delta_f(\Delta(K)$ (where we  identified $\Delta(K)$ and $[0,1]$). A strategy in $\Gamma(z)$ is defined only at $\frac{1}{2}$ since it can be modified elsewhere without altering  the payoff. Therefore an optimal strategy $ \sigma$ in in $\Gamma(z)$ can play anything in $\frac{1}{2}-\epsilon$ and in $\frac{1}{2}+\epsilon$ since no regularity for $\sigma$ is required. The good way to use the proximity between $z$ and $z'$ is to always play as if the initial distribution was $\frac{1}{2}$. Here we have to  define $\sigma'$ such that for all $z\in \Delta(K)$, $\sigma'(z)=\sigma(\frac{1}{2})$.
\end{example}

\begin{lemme}[Splitting procedure]\label{concavite}
The mapping $\hat{v}_{\theta}(z)$ is concave on $\Delta_f(\Delta(K))$.
\end{lemme}
\begin{proof}
We follow the same scheme as for games with incomplete information on one side (see e.g. \cite{msz} Corollary 1.3 p.184). Let $\lambda \in [0,1]$ and let $Y$ be a random variable with values in $\{0,1\}$, such that $\PP(Y=0)=\la$.
Let $z, z' \in \Delta_f(\Delta(K))$. The random variable $P$ is selected according to the distribution $z$ if $Y=0$ and $z'$ if $Y=1$, the state variable $k_1$ is finally selected according to $p$ if $P=p$. Compare now the two following situations: on one hand, the game with initial signals $(Y,P)$ for player $1$ and nothing for player $2$ and on the other hand the game with initial signals $(Y,P)$ for player $1$ and $Y$ for player $2$. These two distributions of initial signals and states fulfill our assumptions and it's clear that the value of the second is less or equal than the value of the first for any evaluation $\theta \in \Delta_f(\NN^*)$ since the set of behavior strategies of player $2$ in the second game is larger than in the first game. Translating this inequality using $\hat{v}$, we deduce directly
\[ \hat{v}_{\theta}(\de_{\lambda z + (1-\lambda){z'}})\geq \hat{v}_{\theta}(\lambda \de_z + (1-\lambda)\de_{z'}) = \lambda \hat{v}_{\theta}(z)+ 1-\lambda \hat{v}_{\theta}(z'),\]
which proves the Lemma.
\end{proof}

\subsection{Auxiliary game $\G$}

Let $X=\De_f(\De(K))$ be the state space, which corresponds to player $2$'s belief about player $1$'s belief about the current state. It is a convex relatively compact subset of a normed vector space and we are going to express the auxiliary game and the recursive formula on this state space.

Let $\G$ be the stochastic game defined by
\begin{itemize}
\item the state space $X=\De_f(\De(K))$,
\item the action space $A=\{f:\De(K)\to \DeI, \text{measurable}\}$ for player $1$,
\item the action space $B=\De( J)$ for player $2$,
\item the payoff function $G:X \times A \times B \rightarrow [0,1]$ defined, for any $z\in X$ by
\[ G(z,a,b)= \sum_{p \in supp(z)} \sum_{(i,j)\in I\times J} b(j)a(p,i)g(p,i,j))z(p),  \]
\noindent where $\mathrm{supp}(z)$ stands for the support of $z$,
\item the transition function  $\ell: X \times A \times B \rightarrow \De_f(X)$ is defined as $\ell(z,a,b)=\Phi(Q(z,a,b))$, where $Q(z,a,b)\in \Delta_f((K) \times (\De(K) \times C) \times ( D))$ is the induced joint distribution of $(k_2,(p,i_1,c_2),(j_1,d_2))$ in the canonical game $\widehat{\Gamma}(\de_z)$ where players play at the first stage $\si_1=a$ and $\tau_1=b$. The sets $C$, $D$, $K$ and $\mathrm{supp}(z)$ being finite and using assumptions $(A1)$ and $(A2)$, we may consider $Q$ as an element in $\Delta_f^*(K \times \NN \times \NN)$.
\end{itemize}

Let us recall the definition of Choquet order on $\Delta_f(X)$.
\begin{definition}
The order $\leq$ on $\Delta_f(X)$ called (reversed) Choquet order is defined by the relation
\[ \mu \leq \nu \Leftrightarrow \; \text{For all continuous concave function on $X$}, \; \tilde{f}(\mu) \leq \tilde{f}(\nu). \]
\end{definition}

We aim to apply a weakened version of Renault \cite{R2012} to the game $\G$, thus let us first recall the hypotheses of the Theorem as they appear in the original article.
\begin{hypothese}\
\begin{itemize}
\item[H1)] The map $\ell$ does not depend on $b$.
\item[H2)] $X$ is a compact convex subset of a normed vector space,
\item[H3)] $A$ and $B$ are convex compact subsets of some topological vector spaces,
\item[H4)] $(a\mapsto G(z,a,b))$ is concave upper semi-continuous $\forall (z,b)\in X\times B$ and  $(b\mapsto G(z,a,b))$ is convex and lower semi-continuous $\forall (z,a)\in X\times A$.
\item[H5)] There exists a subset $\mathcal{C}$ of $1$-Lipschitz functions containing $\phi(1,0)$ such that for all $f$ in $\mathcal{C}$, $\alpha \in[0,1]$, the function $\phi(\alpha,f)$ is in $\mathcal{C}$, where $\phi(\alpha,f)$ is defined by
\[\forall z\in \Delta_f(X)\  \phi(\alpha,f)(z)=\sup_{a\in A} \; \min_{b\in B}\left \{ \alpha G(z,a,b)+(1-\alpha)\tilde{f}(\ell(z,a))\right\}. \]
\item [H6)]  The mapping $a \mapsto \ell(z,a)$ is concave for the Choquet order and continuous.
\item [H7)] (Splitting assumption) Let $z$ be a convex combination in $\Delta_f(\Delta(K))$, $z=\sum_{s=1}^{S} \lambda_s z_s$ and $(a_s)_{s\in S}$ be a family of actions in $A^S$. Then there exists $a\in A$ such that
\[\ell(z,a) \geq \sum_{s\in S} \lambda_s \ell(z_s,a_s) \text{ and  } \min_{b \in B}G(z,a,b) \geq \sum_{s\in S} \lambda_s \min_{b \in B} G(z_s,a_s,b).\]
\end{itemize}
\end{hypothese}

The main consequence of assumption $(A3)$ is that player $2$ cannot influence the transition in the auxiliary game so the map $\ell$ does not depend on $b$, i.e.
\[  \; \forall (z,a) \in X\times A, \, \forall b,b' \in B, \; \ell(z,a,b)=\ell(z,a,b').\]
Thus $(H1)$ is satisfied and from now on, we will work under the shorter notation $l(z,a)$ for $l(z,a,b)$.

The hypotheses $(H2,H3,H4,H6,H7)$ ensure the application of Sion's theorem in several steps of Renault's proof. Here they are not all satisfied since, for example, the set $A$ is not compact. However, it is well known that adding some geometrical hypotheses allows to weaken the topological assumptions in Sion's theorem (see, for instance, Proposition A.$8$ in Sorin's monography \cite{sorinbook}). For instance, if $A$ is a convex set, $B$ is a compact convex subset of a topological vector space, $(a\mapsto G(z,a,b))$ is concave $\forall (z,b)\in X\times B$ and $(b\mapsto G(z,a,b))$ is convex and lower semi-continuous $\forall (z,a)\in X\times A$, Sion's result applies to the one-stage game: the game $\G_1(z)$ has a value. They can be replaced without altering the proof by the following hypotheses.

\begin{hypothese}\
\begin{itemize}
\item[H2')] $X$ is a relatively compact convex subset of a normed vector space.
\item[H3')] $B$ is a convex compact subset of a topological vector space, $A$ is a convex set.
\item[H4')] $(a\mapsto G(z,a,b))$ is concave $\forall (z,b)\in X\times B$ and  $(b\mapsto G(z,a,b))$ is convex and lower semi-continuous $\forall (z,a)\in X\times A$.
\item [H6')]  The mapping $a \mapsto \ell(z,a)$ is concave for the Choquet order.
\end{itemize}
\end{hypothese}

Assumption $(H2')$ is satisfied since the Wasserstein distance can be extended to a norm on the space of finite signed measures.
Moreover assumptions $(H3')$ and $(H4')$ are clearly satisfied. Therefore, we need to prove $(H6')$ and $(H7)$.

\begin{lemme} The game $\G$ fulfills $H6'$ and $H7$.
\end{lemme}
\begin{proof}
Let $z$ be a convex combination in $X$, $z=\sum_{s=1}^{S} \lambda_s z_s$ and $(a_s)_{s\in S}$ be a family of actions in $A^S$. Denote $\mu(z_s,a_s) \in \Delta_f(\Delta(K)\times  I)$ the joint law induced on $\Delta(K)\times  I$ by $(z_s,a_s)$. By disintegration, there exists $a\in A$ such that $\mu(z,a)=\sum_{s\in S} \lambda_s \mu(z_s,a_s)$. A first, note that $Q(z,a,b)= \sum_{s\in S}\lambda_s Q(z_s,a_s,b)$.

Given $(z,a,b)$, we consider the canonical game $\hat{\Gamma}(z)$. In this game, a pair $(k,p)$ is chosen according to the probability $\pi \in \Delta_f(K \times \Delta(K))$ defined by $\pi(k,p)=p^kz(p)$ for all $(k,p) \in K \times \Delta(K)$. Then, player $1$ receives the signal $c_1=p$ and player $2$ receives no initial signal. We associate to $(a,b)$ a pair strategies for the first stage $(\sigma_1,\tau_1)$ by $\sigma_1(p)=a(p)$ and $\tau_1=b$. Then, $Q(z,a,b)$ denotes the joint distribution of $(k_2,(p,i_1,c_2),(j_1,d_2))$. Since the conditional law of $(k_2,c_2,d_2)$ given $(p,i_1,j_1)$ is $q(p,i_1,j_1)$, it follows that $Q(z,a,b)$ is bilinear with respect to $(\mu(z,a),b)$ (with abusive notations).

Let $\rho=Q(z,a,b)$ (resp. $\rho_s=Q(z_s,a_s,b)$) and $c'=(p,i_1,c_2)$ (resp. $d'=(j_1,d_2)$). Let $C'= (\cup_{s\in S} supp(z_s)) \times I \times C$ and $D'=J\times D$.
 By construction,
\[ \ell(z,a) =\Phi(\rho)= \CL_{\rho}(\CL_{\rho}( \CL_{\rho}(k_2 \mid c') \mid d') )= \sum_{d' \in D'} \rho(d')\de_{\CL_{\rho}( \CL_{\rho}(k_2 \mid c') \mid d')}.\]
Using Lemma \ref{eqbelief}, we have the following equality $\rho$-almost surely
\[ \CL_{\rho}(k_2 \mid p,i_1,c_1)= F(p,i_1,c_2).\]
This implies that
\[ \CL_{\rho}(\CL_{\rho}(k_2 \mid c'),d')=\CL_{\rho}(F(c'),d') \in \De_f( C'' \times D') \]
where $C''=F(C')$. A similar equality holds with $\rho_s$ instead of $\rho$ for all $s\in S$. By definition of $l(z,a)$, we deduce that
\[ \ell(z,a)=\CL_{\rho}(\CL_{\rho}(F(c')\mid d'))= \Psi(\CL_{\rho}(F(c'),d')), \]
where $\Psi$ is the disintegration map defined by
\[ \Psi : \Delta( C'' \times D') \rightarrow \Delta_f( \Delta(C'')) : m\rightarrow \sum_{d' \in D'} m(d') \delta_{\CL_m(c''|d')}, \]
where $\CL_m(c''|d')$ denoted the conditional law of $c''$ given $d'$. It was proved in Renault \cite{R2012} (Lemma 4.16) that $\Psi$ is concave for the Choquet order on $\Delta_f(\Delta(C''))$. However, $C''$ being a finite subset of $\Delta(K)$, $\Delta(C'')$  is identified as a compact convex subset of $X$. It follows easily that the convex order on $\Delta_f(\Delta(C''))$ coincides with the order induced by the convex order on $\Delta_f(X)$.

We conclude that the first part of H7 holds since
\[ l(z,a) = \Psi (\sum_{s\in S} \lambda_s\CL_{\rho_s}(F(c'),d')) \geq \sum_{s\in S} \lambda_s \Psi(\CL_{\rho_s}(F(c'),d')) = \sum_{s\in S}\lambda_s l(z_s,a_s). \]
For the second part of H7, it is sufficient to note that (again with abusive notations) $\mu(z,a) \mapsto G(z,a,b)$ is linear so that for all $b\in B$
\[  G(z,a,b) = \sum_{s\in S} \lambda_s G(z_s,a_s,b),\]
which implies the result. Finally, in case $z_s=z$ for all $s$, the same arguments also imply $(H6')$ since in this case one can choose $a=\sum_{s\in S}\lambda_s a_s$ in the above proof.
\end{proof}

The proof of the following proposition follows from Proposition 3.21 in Renault \cite{R2012}.
\begin{proposition}
Assuming $(H1,H2',H3',H4',H6',H7)$, then for any $\theta \in \De_f(\NN^*)$ and any $\eta \in \Delta_f(X)$, the game $\G_{\theta}(\eta)$ has a value $w_{\theta}(\eta)$ such that
\begin{align} \forall z \in X, w_{\theta}(z)&=
\sup_{a\in A} \; \min_{b\in B}\left \{ \theta_1 G(z,a,b)+(1-\theta_1)w_{\theta^+}(\ell(z,a))\right\},\\
&= \min_{b\in B} \sup_{a\in A} \left \{ \theta_1 G(z,a,b)+(1-\theta_1)w_{\theta^+}(\ell(z,a))\right\},
\end{align}
where $\theta^+$ is defined by $\theta^+_t= \frac{\theta_{t+1}}{\sum_{m \geq 2} \theta_m}$ for $t\geq 1$ whenever $\sum_{m \geq 2} \theta_m >0$ and is defined arbitrarily otherwise. Moreover, in $\G_{\theta}(\eta)$, player $1$ has $\varepsilon$-optimal Markov strategies for all $\varepsilon>0$ and player $2$ has optimal Markov strategies.
\end{proposition}

In order to prove the last assumption $(H5)$, we first prove that the value of the game $\mathcal{G}_{\theta}$ is equal to the canonical value function $\hat{v}_{\theta}$. Since we proved that the canonical value is $1$-Lipschitz, it will imply using the previous Proposition that the set of functions $\mathcal{C}=\{v_\theta, \theta \in \Delta_f(\NN^*)\}$ satisfies $(H5)$.

We now prove that the value functions of both games are the same. The proof is classic and consists to show that both families of functions are linked by the same recursive formula.
\begin{proposition}\label{val_eq}
For all $\theta\in \De_f(\NN^*)$ and for any $z \in X$, $w_{\theta}(z)=\hat{v}_{\theta}(z)$.
\end{proposition}

\begin{corollaire}
The game $\G$ fulfills $(H5)$.
\end{corollaire}

\begin{proof}[Proof of Proposition \ref{val_eq}]
Notice first that $\hat{v}_1(z)=w_1(z)$ for all $z$. This comes indeed almost from the definition
\begin{align*}
\hat{v}_1(z) & = \sup_{\sigma_1:\De(K) \rightarrow \De(I)} \inf_{b\in \De(J)} \int_{\De(K)} g(p,\sigma_1(p),b) dz(p) \\
       &= \sup_{a \in A} \min_{b\in \De(J)} G(z,a,b) \\
       &= \min_{b\in \De(J)} \sup_{a \in A} G(z,a,b)\\
       &= w_1(z).
\end{align*}
It is enough to prove that $w$ and $\hat{v}$ satisfy the same recurrence formula. We will prove that $\hat{v}$ satisfies the recurrence formula in $\G$, i.e.
\begin{align*}
\hat{v}_{\theta}(z) &= \sup_{a\in A} \min_{b \in B}  \theta_1 G(z,a,b) + (1-\theta_1) \hat{v}_{\theta^+}(\ell(z,a)) \\
                 &= \min_{b\in B} \sup_{a \in A} \ \theta_1 G(z,a,b) + (1-\theta_1) \hat{v}_{\theta^+}(\ell(z,a)) .
\end{align*}
We prove the recursive formula by induction on the greatest element in the support of $\theta$. If $\theta=\de_1$, it follows from the preceding equality. Fix now $n \geq 2$, and assume that the proposition is true for every $\theta$ supported by $\{1,...,n-1\}$. Let $z\in \De_f(\De(K))$.
We first prove that player $1$ can defend in $\hat{\Gamma}_{\theta}(z)$ the quantity
\[\min_{b\in B} \sup_{a \in A} \left(\theta_1 G(z,a,b) + (1-\theta_1) \hat{v}_{\theta^+}(\ell(z,a)) \right).\]
Using the canonical representation $\widehat{\Gamma}$, $\hat{v}_{\theta}(z)=v_{\theta}(\pi)$ where $\pi \in \De_f(K \times \De(K) \times \De_f(\De(K)))$ is defined by
\[\forall (k,p,x)\in K \times \De(K) \times \De_f(\De(K)), \; \pi(k,p,x)=p(k)z(p)\ind_{x=z}.\]
Consider the game $\Ga_{\theta}(\pi)$. Let $\varepsilon >0$ and $\tau$ be a strategy of player $2$. Denoting by $b$ the law induced by $\tau_1$, let  $a^*\in A$ an action which realizes the supremum up to $\varepsilon$ in the expression
\[\theta_1 G(z,a,b) + (1-\theta_1) \hat{v}_{\theta^+}(\ell(z,a)).\]
Let $\sigma^*$ be an $\varepsilon$-optimal strategy in the game $\Gamma_{\theta^+}(Q(z,a^*,b))$. Define then $\sigma$ by $\sigma_1=a^*$ and for all $n\in \NN^*$, $h^I_n=(p,i_1,c_2,...,i_{n-1},c_{n})$, $\sigma_n(h^I_n)=\sigma^*_{n-1}(c',h^{1,+}_{n-1})$ where $c'=(p,i_1,c_2)$ and $h^{1,+}_{n-1}=(i_2,c_3,..i_{n-1},c_n)$. We have
\[\ga_{\theta}(\mu,\si,\tau)= \theta_1 G(z,a^*,b) + (1-\theta_1) \ga_{\theta^+}(Q(z,a^*,b),\sigma^*,\tau^+),\]
where $\tau^+$ is a continuation strategy. Precisely, for all $n \in \NN^*$, $\tau^+_{n-1}(d',h^{2,+}_{n-1})=\tau_{n}(h^{II}_n)$ with $h^{2,+}_{n-1}= (j_2,d_3,..,j_{n-1},d_n)$, $h^{II}_n=(d',h^{2,+}_{n-1})$ and $d'=(j_1,d_2)$ is the ``signal'' for player $2$ given by $Q(z,a^*,b)$.
\p
Therefore, $\sigma^*$ and $\tau^+$ can be seen as behavior strategies in a new game with initial signals corresponding to the past history in the original game and since $\sigma^*$ is $\varepsilon$-optimal in $\Gamma(Q(z,a^*,b))$, we have
\begin{align*}
\ga_{\theta}(\mu,\si,\tau) & \geq \theta_1 G(z,a^*,b) + (1-\theta_1) v_{\theta^+}(Q(z,a^*,b))-\varepsilon \\
                             & \geq \sup_{a\in A} \theta_1 G(z,a,b) + (1-\theta_1) v_{\theta^+}(Q(z,a,b)) -2\varepsilon \\
                             & = \sup_{a\in A} \theta_1 G(z,a,b) + (1-\theta_1) \hat{v}_{\theta^+}(\ell(z,a)) - 2 \varepsilon \\
                             & \geq \min_{b\in B} \sup_{a\in A} \theta_1 G(z,a,b) + (1-\theta_1) \hat{v}_{\theta^+}(\ell(z,a)) - 2 \varepsilon.
\end{align*}
It follows that $\hat{v}_{\theta}(z) \geq \min_{b\in B} \sup_{a\in A} \theta_1 G(z,a^*,b) + (1-\theta_1) \hat{v}_{\theta^+}(\ell(z,a))$ by sending $\varepsilon$ to zero.
\p
Let us show that player $2$ can defend $\sup_{a\in A} \min_{b\in B} (\theta_1 G(z,a,b)+ (1-\theta_1) \hat{v}_{\theta^+}(\ell(z,a))$ in $\Ga({\mu})$. Fix a strategy $\sigma$ of player $1$ and let $a=\sigma_1$, there exists $b^*\in B$ achieving $\min_b G(z,a,b)$. We also choose $\tau^*$ an optimal strategy for player $2$ in the game $\Ga_{\theta^+}(Q(z,a,b^*))$. This defines a strategy $\tau$ such that
\begin{align*}
\gamma_{\theta}^{\mu}(\sigma,\tau) & = \theta_1 G(z,a,b^*) + (1-\theta_1) \gamma_{\theta^+}^{Q(z,a,b^*)}(\sigma^+,\tau^*)\\
& \leq  \theta_1 G(z,a,b^*) + (1-\theta_1) v_{\theta^+}(Q(z,a,b^*)) \\
& =  \theta_1 G(z,a,b^*) + (1-\theta_1) \hat{v}_{\theta^+}(\ell(z,a)) \\
& = \min_{b\in B} \theta_1 G(z,a,b) + (1-\theta_1) \hat{v}_{\theta^+}(\ell(z,a)).
\end{align*}
Thus $\hat{v}_{\theta}(z) \leq \sup_{a\in A} \min_{b\in B} (\theta_1 G(z,a,b)+ (1-\theta_1) \hat{v}_{\theta^+}(\ell(z,a))$. Finally, since the maxmin is always smaller than the minmax, all the intermediate inequalities are equalities.
\end{proof}

\subsection{Existence of the uniform value}
Let us at first recall the first main result proved in \cite{R2012} which holds under our set of weakened assumptions.
\begin{theoreme}[Renault(2012)]
Assume that $H1,H'2,H'3,H'4,H5,H'6,H7$ hold. Then for every initial distribution $\eta \in \De_f(X)$, the game has a uniform value $w^*(\eta)$.
\noindent Moreover player $1$ can guarantee $w^*(\eta)$ with a Markov strategy:
\[\forall \epsilon>0,\ \exists \sigma \in \Sigma^M, \exists N_0 \in \NN, \ \forall N\geq N_0 \ \forall \tau' \in \Tau, \ \gamma_N(\eta,\sigma,\tau')\geq w^*(\eta)-\epsilon.\]
and we have $w^*(\eta)=\inf_{n\geq 1} \sup_{m\geq 0} w_{m,n}(\eta)$.
\end{theoreme}
In order to conclude the proof, we show that both players can guarantee
\[v^*(\pi)=\inf_{n\geq 1} \sup_{m \geq 0} v_{m,n}(\pi),\]
where $v_{m,n}(\pi)=v_{\theta_{m,n}}(\pi)$ and $\theta_{m,n}$ is the uniform law between stage $m$ and $m+n$.
\p
The game $\mathcal{G}(z)$ satisfies assumptions $H1,...,H7'$ so it has a uniform value given by
\[ w^*(z)=\inf_{n\geq 1} \sup_{m\geq 0} w_{m,n}(z).\]
And by proposition \ref{val_eq}, the value in $\G$ and in the reduced game are equal, so if $\pi \in \De^*_f(K \times \NN \times \NN)$ we have
\[v^*(\pi)=\inf_{n\geq 1} \sup_{m\geq 0} v_{m,n}(\pi) = \inf_{n\geq 1} \sup_{m\geq 0} \hat{v}_{m,n}(\Phi(\pi))= \inf_{n\geq 1} \sup_{m\geq 0} w_{m,n}(\Phi(\pi))=w^*(\Phi(\pi)),\]
Thus player $1$ can guarantee $v^*(\pi)$ in $\mathcal{G}(\Phi(\pi))$ with a Markov strategy. Let us check that he can guarantee $v^*(\pi)$ in the game $\widehat{\Gamma}(\Phi(\pi))$ or equivalently in $\Gamma(\pi)$.

\begin{proposition}\label{markovian}
Any Markovian strategy $\widehat{\sigma}$ of player $1$ in $\G_{\infty}(z)$ induces a strategy $\si$ in $\hat{\Gamma}_{\infty}(z)$  guaranteeing the same amount.
\end{proposition}
\begin{proof} Let $\widehat{\si}$ be a behavior strategy in $\G_{\infty}(z)$. Let us describe the strategy $\si$. Player $1$ plays at the first round in $\Gamma_{\infty}(z)$ the mixed action $\widehat{\si}_1(z)(p)$  where $p$ is his initial signal. Then, at round $n$, he plays the mixed action $\widehat{\si}_n(y_n)(x_n)$.
That this strategy is a well-defined strategy follows from Lemma \ref{markovstrat}.

It remains to prove that this strategy guarantees the same quantity as $\widehat{\si}$. Let us fix $n\in \NN^*$, we will prove that there exists a best reply $\widetilde{\tau}$ to $\si$ in $\widehat{\Gamma}_{n}(z)$ which can be seen as a strategy $\hat{\tau}$ in $\G_{\theta}(z)$ and such that
\[ \gamma_\theta(z,\si,\widetilde{\tau}) = \widehat{\gamma}_\theta(z,\widehat{\si},\widehat{\tau}).\]
We will proceed by backward induction. Let us fix a best reply $\tau$ to $\si$ in $\widehat{\Gamma}_{\theta}(z)$. We will construct a strategy $\widetilde{\tau}$ which depends at stage $m$ on $h^{II}_m$ only through $y_m$. Recall that $\si$ is fixed so that $y_m(h^{II}_m)$ can be computed by player $2$. At first let us replace $\tau_n$ by 
\[ \widetilde{\tau_n}(y_n) = \EE_{\PP^z_{\si\tau}}[ \tau(h^{II}_n) \mid y_n ]. \]
Note that this conditional expectation depends on the strategies $\si,\tau$ up to stage $n-1$. Let us prove that the payoff at the last stage $n$ is not modified.
\begin{align*}
\EE_{\PP^z_{\si\tau}}[ g(k_n,i_n,j_n)] &= \EE_{\PP^z_{\si\tau}}[\EE_{\PP^z_{\si\tau}}[ g(k_n,i_n,j_n) \mid h^I_n,h^{II}_n]] \\
&= \EE_{\PP^z_{\si\tau}}[\gamma_1(x_n,\si_n(y_n,x_n),\tau_n(h^{II}_n))] \\
&= \EE_{\PP^z_{\si\tau}}[\EE_{\PP^z_{\si\tau}}[ \gamma_1(x_n,\si_n(y_n,x_n),\tau_n(h^{II}_n)) \mid h^{II}_n]] \\
&= \EE_{\PP^z_{\si\tau}}[ \int \gamma_1(x,\si_n(y_n,x),\tau_n(h^{II}_n)) dy_n[x]] \\
&=\EE_{\PP^z_{\si\tau}}[\EE_{\PP^z_{\si\tau}}[\int \gamma_1(x,\si_n(y_n,x),\tau_n(h^{II}_n)) dy_n[x] \mid y_n]] \\
&= \EE_{\PP^z_{\si\tau}}[ \int \gamma_1(x,\si_n(y_n,x),\EE_{\PP^z_{\si\tau}}[\tau_n(h^{II}_n)\mid y_n]) dy_n[x] ]\\
&= \EE_{\PP^z_{\si\tau}}[ \int \gamma_1(x,\si_n(y_n,x),\widetilde{\tau}_n(y_n)) dy_n[x] ]\\
&= \EE_{\PP^z_{\si,(\tau_1,...,\tau_{n-1},\widetilde{\tau}_n)}}[ g(k_n,i_n,j_n)].
\end{align*}
The above equations show that the expected payoff at stage $n$ when player $2$ is playing the best reply  $(\tau_1,...,\tau_{n-1},\widetilde{\tau}_n)$ against $\si$ is a function of $\si$ and of the law of $y_n$. Assume now that at step $m$, we have proved that there exists a best reply to $\si$ of player $2$ such that the sum of expected payoffs for the stages $m+1,...,n$ is a function of $\si$ and of the law of $(y_{m+1},..,y_n)$ only. We can replace $\tau_m(h^{II}_m)$ by $\widetilde{\tau_m}(y_m) = \EE_{\PP^z_{\si\tau}}[\tau_m(h^{II}_m)\mid y_m ]$ without modifying the expected payoff of stage $m$ with the same argument as above. Using assumption $(A3)$, Lemma \ref{markovstrat} and the definition of $\sigma$, the law of $(y_{m+1},...y_n)$ is not modified by this operation which proves that this modified strategy is still a best reply to $\si$.
\end{proof}

Secondly, we prove that Player $2$ can guarantee $v^*(\pi)$ by splitting the stage in blocks and playing on each block separately since he has no influence on the transition. The following results are quite similar to the corresponding ones proved in Renault \cite{R2012} and are reproduced here since their proofs are very short.
\begin{lemme}
For every $\pi \in \De^*_f(K\times C'\times D')$, $n\geq 1$ and $m\geq 1$, $\forall \ \tau_1,...,\tau_m$, $\exists \, \tau_{m+1},...,\tau_{m+n}$ such that the  strategy $\tau_1,...,\tau_m,...,\tau_{m+n}$ of player $2$ is optimal in the game $\Gamma_{m,n}$.
\end{lemme}

\begin{proof}
Let $\pi \in \De(K\times C'\times D')$, $n\geq 1$ and $m\geq 0$, and $\tau_1,...,\tau_m$ such that $\tau_i:D' \times ( J \times D)^{i-1} \rightarrow \Delta(J)$. We define $\T^*$ the subset strategies of player $2$ which start with $\tau_1,...,\tau_m$ and we consider the game with the evaluation $\theta_{m,n}$ and the set of strategies $\Sigma$ and $\T^*$. It can be seen as the mixed extension of a finite game, thus the value exists and will be denoted $v^*_{m,n}(\pi)$. Since the set of strategies of player $2$ is smaller than $\T$, we have $v_{m,n}(\pi) \leq v^*_{m,n}(\pi)$. But using the same method as for proving the recursive formula of Proposition \ref{val_eq}, for any $\sigma$, we can build a strategy which defends $v_{m,n}(\pi)$. Both values are therefore equal and any optimal strategy in the restricted game satisfies the conclusion of the lemma.
\end{proof}

\begin{proposition}
For every $\pi \in \De^*_f(K\times C'\times D')$, player $2$ can guarantee $v^*(\pi)$ in the game $\Gamma_{\infty}(\pi)$.
\end{proposition}

\begin{proof}
We prove that for all $n \in \NN$, Player $2$ can guarantee the payoff $\sup_{m\geq0} v_{m,n}(p)$. Let $n\in \NN$ be a number of stages, then for each $L\in \NN$ we split the game of length $nL$ in $L$ blocks of length $n$: $B_1,...,B_L.$  We define the strategy $\tau^*$ by induction on the block.

Let $\tau$ be an optimal strategy in $\Gamma_{1,n}(\pi)$ then we set $\tau^*_i=\tau_i$ for all $i\in \{1,..,n\}$. Once we have constructed $\tau^*_1,...,\tau^*_{nl}$ for some $1 \leq l \leq L-1$, we define the game $\Gamma^{\#}_{nL+1,n}(\pi)$ where the player $2$ has to play $\tau^*_i$ for all $i\leq nl$. We have $v^{\#}_{nL+1,n}(\pi) = v_{nL+1,n}(\pi)$ using the preceding Lemma.  Let $\tau$ be an optimal strategy in $\Gamma^{\#}_{nL+1,n}(\pi)$ and set $\tau^*_i=\tau_i$ for all $i\in \{nL+1,...,(n+1)L\}$.
We have
\begin{align*}
\gamma_{Ln}(\sigma,\tau^*) &=\frac{1}{nL} \EE^{\pi}_{\sigma \tau^*} \left( \sum_{m=0}^{Ln} g(k_m,i_m,j_m) \right) =\frac{1}{nL} \sum_{d=0}^{L-1} \EE^{\pi}_{\sigma \tau^*} \left (\sum_{m=dn+1}^{(d+1)n} g(k_m,i_m,j_m) \right) \\
                           &\leq \frac{1}{L} \sum_{d=0}^{L-1} v_{dn+1,n}(\pi) \leq \frac{1}{L} \sum_{d=0}^{L-1} \sup_{m\geq 0} v_{m,n}(\pi) \\
                           &\leq \sup_{m\geq 0} v_{m,n}(\pi). \\
\end{align*}
The payoff being bounded, we deduce that this strategy guarantees $v_{m,n}(\pi)$. Finally, Player $2$ can guarantee the minimum on $n\in \NN$, $\inf_{n\in \NN} \sup_{m\geq 0} v_{m,n}(\pi)=v^*(\pi).$
\end{proof}

Since each player can guarantee $v^*(\pi)$, the game has a uniform value given by $v^*(\pi)$ which concludes the proof of Theorem \ref{main1}.

\end{document}